\documentclass[a4paper,11pt,twoside]{amsart}

\usepackage{amsmath}
\usepackage{amsthm}
\usepackage{amssymb}
\usepackage[all]{xy}
\usepackage{hyperref}
\usepackage{rotating}

\newtheorem{thm}{Theorem}[section]
\newtheorem{prop}[thm]{Proposition}
\newtheorem{lem}[thm]{Lemma}
\newtheorem{cor}[thm]{Corollary}

\numberwithin{equation}{section}

\theoremstyle{definition}
\newtheorem{definition}[thm]{Definition}
\newtheorem{remark}[thm]{Remark}
\newtheorem{ex}[thm]{Example}

\newcommand{\rote}{{\begin{turn}{270}$\!\!\!\epsilon\ $\end{turn}}}
\newcommand{\OO}{{\rm O}}
\newcommand{\Db}{{\rm D}^{\rm b}}
\newcommand{\D}{{\rm D}}

\newcommand{\Km}{{\rm Km}}

\newcommand{\Aut}{{\rm Aut}}

\newcommand{\Pic}{{\rm Pic}}
\newcommand{\ch}{{\rm ch}}

\newcommand{\coh}{{\cat{Coh}}}
\newcommand{\qcoh}{{\cat{QCoh}}}

\newcommand{\mmod}[1]{{\ka_{#1}\text{-}\cat{Mod}}}

\newcommand{\Hom}{{\rm Hom}}

\newcommand{\id}{{\rm id}}

\newcommand{\dual}{^{\vee}}

\newcommand{\mor}[1][]{\xrightarrow{#1}}
\newcommand{\isomor}{\mor[\sim]}

\newcommand{\cat}[1]{\begin{bf}#1\end{bf}}
\newcommand{\Ext}{{\rm Ext}}
\newcommand{\FM}[1]{\Phi_{#1}}

\newcommand{\lotimes}{\stackrel{\mathbf{L}}{\otimes}}

\newcommand{\HH}{\mathrm{H\!H}}
\newcommand{\HO}{\mathrm{H}\Omega}
\newcommand{\HT}{\mathrm{HT}}
\newcommand{\Htil}{\widetilde H}
\newcommand{\Dp}{\mathrm{D_{perf}}}
\newcommand{\Inf}{{\rm Inf}}

\newcommand{\cal}{\mathcal}
\newcommand{\ka}{{\cal A}}
\newcommand{\kb}{{\cal B}}

\newcommand{\ke}{{\cal E}}
\newcommand{\kf}{{\cal F}}
\newcommand{\kg}{{\cal G}}
\newcommand{\kh}{{\cal H}}
\newcommand{\kk}{{\cal K}}
\newcommand{\ki}{{\cal I}}
\newcommand{\kl}{{\cal L}}

\newcommand{\ko}{{\cal O}}
\newcommand{\kp}{{\cal P}}

\newcommand{\kt}{{\cal T}}

\newcommand{\kz}{{\cal Z}}
\newcommand{\NN}{\mathbb{N}}
\newcommand{\ZZ}{\mathbb{Z}}

\newcommand{\RR}{\mathbb{R}}
\newcommand{\CC}{\mathbb{C}}

\newcommand{\RHom}{\mathbf{R}\Hom}
\newcommand{\RGamma}{\mathbf{R}\Gamma}
\newcommand{\Lotimes}{\stackrel{\mathbf{L}}{\otimes}}
\newcommand{\aR}{\mathbf{R}}
\newcommand{\eL}{\mathbf{L}}
\newcommand{\fb}{\overline{f}}
\newcommand{\tX}{\widetilde{X}}
\newcommand{\tY}{\widetilde{Y}}

\def\aa{\alpha}
\def\bb{\beta}

\def\dd{\delta}
\def\Dd{\Delta}

\def\ix{\iota_X}
\def\iy{\iota_Y}

\def\tt{\tau}

\def\Gg{\Gamma}

\begin{document}

\title{Infinitesimal Derived Torelli Theorem for K3 surfaces}

\author[E.\ Macr\`i and P.\ Stellari]{Emanuele Macr\`i and Paolo Stellari}

\address{E.M.: Department of Mathematics, University of Utah, 155 South 1400 East, Salt Lake City, UT 84112, USA}
\email{macri@math.utah.edu}

\address{S.M.: Department of Mathematics and Statistics, University of Massachusetts--Amherst, 710 N. Pleasant Street, Amherst,
MA 010003, USA}\email{mehrotra@math.umass.edu}

\address{P.S.: Dipartimento di Matematica ``F. Enriques'',
Universit{\`a} degli Studi di Milano, Via Cesare Saldini 50, 20133
Milano, Italy}
\email{paolo.stellari@unimi.it}

\keywords{K3 surfaces, derived categories, deformations}

\subjclass[2000]{18E30, 14J28, 13D10}

\begin{abstract}
We prove that the first order deformations of two smooth projective K3 surfaces are derived equivalent under a Fourier--Mukai transform if and only if there exists a special isometry of the total cohomology groups of the surfaces which preserves the Mukai pairing, an infinitesimal weight-$2$ decomposition and the orientation of a positive $4$-dimensional space. This generalizes the derived version of the Torelli Theorem. Along the way we show the compatibility of the actions on Hochschild homology and singular cohomology of any Fourier--Mukai functor.
\end{abstract}

\maketitle

\section{Introduction}\label{sec:Intro}

A great deal of geometric information is encoded in the lattice and Hodge structures defined on the cohomology groups of a K3 surface (i.e.\ a smooth, simply-connected, projective surface with trivial canonical bundle). Just two results making this
plain are the classical Torelli Theorem (see \cite{BR,LP,PS}) and its more recent categorical version,
the Derived Torelli Theorem.

This latter theorem, in the final form resulting from the
combination of \cite{Mu,Or1} and \cite{HMS1}, asserts that any
equivalence between the derived categories of coherent sheaves of
two K3 surfaces induces a Hodge isometry on their cohomology
groups which preserves the orientation of some positive
four-space. The reverse implication is also true and follows from
a detailed analysis of the geometry of moduli spaces of stable
sheaves on K3 surfaces.

Since the deformation theory of K3 surfaces is well understood, one can wonder if, in an appropriate setting, the Derived Torelli Theorem can be extended to (at least) first order deformations. More precisely, one can ask if the equivalences of the derived categories of first order deformations of K3 surfaces can be detected by the existence of isometries of some kind of deformed lattice and Hodge structures on the total cohomology groups.

\smallskip

To state our answer to this question, we need to sketch briefly the categorical setting and some additional structures on the total cohomologies of K3 surfaces (which will be extensively described in Sections \ref{subsec:category} and \ref{subsec:proof2} respectively).

For a smooth projective (complex) variety $X$, all the abelian categories which are first order deformations of the abelian category $\coh(X)$ of coherent sheaves on $X$ are parametrized by the second Hochschild cohomology group $\HH^2(X)$. In particular, for each element $v\in\HH^2(X)$, one can produce (see \cite{Toda}) an abelian category $\coh(X,v)$ which is the first order deformation of $\coh(X)$ in the direction $v$. The kernels of the Fourier--Mukai functors between the derived categories $\Db(X_1,v_1)$ and $\Db(X_2,v_2)$ of $\coh(X_1,v_1)$ and $\coh(X_2,v_2)$ are perfect complexes (i.e.\ complexes locally quasi-isomorphic to a finite complex of locally free sheaves) in $\Dp(X_1\times X_2,-J(v_1)\boxplus v_2)$. Roughly speaking, the operator $J$ changes the sign of a component of $v_1$.

If $X$ is a K3 surface, the total cohomology group $H^*(X,\ZZ)$ tensored by $\ZZ[\epsilon]/(\epsilon^2)$ inherits a pairing and, chosen $v\in\HH^2(X)$,  a weight-$2$ decomposition (via the Hochschild--Kostant--Rosenberg isomorphism). Such a $\ZZ[\epsilon]/(\epsilon^2)$-module, endowed with these additional structures, is called \emph{infinitesimal Mukai lattice} and it is denoted by $\widetilde H(X,v,\ZZ)$ (see Definition \ref{def:Mukailattice}). This lattice is closely related to the one defined in \cite{HS} for the case of twisted K3 surfaces. The analogy is actually quite precise, since a K3 surface twisted by an element in its Brauer group can be thought as a complete (not just formal) deformation in a gerby direction.

An isomorphism $g:\widetilde H(X_1,v_1,\ZZ)\isomor\widetilde H(X_2,v_2,\ZZ)$ of $\ZZ[\epsilon]/(\epsilon^2)$-modules is a Hodge isometry if it preserves the pairings and the weight-$2$ decompositions. The Hodge isometies we will consider are the effective and orientation preserving ones, i.e.\ those induced by isomorphisms $H^*(X_1,\ZZ)\isomor H^*(X_2,\ZZ)$ preserving the Mukai pairings, the natural weight-$2$ Hodge structures and the orientation of the spaces generated by the four positive directions in $H^*(X_1,\ZZ)$ and $H^*(X_2,\ZZ)$ (see the beginning of Section \ref{subsec:proof2}).

\smallskip

With this notation in mind we can state the main result of this paper, which is an infinitesimal version of the Derived Torelli Theorem.

\begin{thm}\label{thm:main1} Let $X_1$ and $X_2$ be smooth complex projective K3 surfaces and let $v_i\in\HH^2(X_i)$, with $i=1,2$. Then the following are equivalent:
    \begin{itemize}
    \item[(i)] There exists a Fourier--Mukai equivalence $$\FM{\widetilde\ke}:\Db(X_1,v_1)\isomor\Db(X_2,v_2)$$ with $\widetilde\ke\in\Dp(X_1\times X_2,-J(v_1)\boxplus v_2)$.
    \item[(ii)] There exists an orientation preserving effective Hodge isometry $$g:\Htil(X_1,v_1,\ZZ)\isomor\Htil(X_2,v_2,\ZZ).$$
    \end{itemize}
\end{thm}

As in the classical case, this result can be reformulated in a neat way for autoequivalences of Fourier--Mukai type. Indeed, Theorem \ref{thm:main1} yields the existence of a surjective group homomorphism
\[
\xymatrix{\Aut^\mathrm{FM}(\Db(X,v))\ar@{->>}[r]&\OO_+(\Htil(X,v,\ZZ))},
\]
where $\Aut^\mathrm{FM}(\Db(X,v))$ is the group of equivalences of Fourier--Mukai type of the triandulated category $\Db(X,v)$ and $\OO_+(\Htil(X,v,\ZZ))$ is the group of orientation
preserving effective Hodge isometries. In the non-deformed case,
Bridgeland gave in \cite{B} a very nice conjectural description of
the kernel of the previous morphism. As we will observe in Section
\ref{subsec:auto}, the same applies to the infinitesimal case.

\smallskip

The final goal of our investigation would be to generalize Theorem
\ref{thm:main1} to deformations of K3 surfaces of any order and
possibly formal. A key step in this direction would be to show
that one can deform to any order the abelian categories of
coherent sheaves and, compatibly, the kernel of any Fourier--Mukai
equivalence. There are already examples in the literature of
deformations in this broader generality. This is the case of the
Poincar\'{e} sheaf for an abelian variety and its dual (see
\cite{BBT}). Unfortunately, the argument there seems to be quite
ad hoc and we cannot hope to apply those techniques to the
situation we want to treat. A more general attempt to deal with
this problem for abelian varieties has been pursued by D.\
Arinkin. Deformations of kernels of Fourier--Mukai equivalences
were also studied in \cite{HMS1} for very special analytic
directions in the case of K3 surfaces. For first order
deformations, such a theory, which will be used in this paper, has
been completely carried out in \cite{Toda} (see Theorem
\ref{thm:Toda}).

Once this goal is achived, one should be able to define a deformation functor $\mathrm{DF}_1$, over local Artin algebras, of
families of derived categories of non-commutative and gerby K3 surfaces. The definition of $\mathrm{DF}_1$, for first-order-deformations, is the content of \cite{Toda}.

On the other hand, there exists a second deformation functor $\mathrm{DF}_2$ of variations of weight-$2$ Hodge structures of
the Mukai lattice, endowed with the Mukai pairing. This is because, there is a global versal analytic
deformation, which is local analytically mini-versal, provided by a period domain.
Once an appropriate definition of $\mathrm{DF}_1$ is given, one should also have a morphism of functors
$p:\mathrm{DF}_1\to\mathrm{DF}_2$, which associates to a family of generalized K3-surfaces, a variation of Hodge structures.

In this setting, Theorem \ref{thm:main1} should state that $p$ induces isomorphisms of the Zariski tangent spaces
of the two deformation functors. A generalization of this result to deformations of K3 surfaces of any order or to formal deformations would allow to prove that $p$ is actually an isomorphism of functors.

\medskip

To prove Theorem \ref{thm:main1} we will need to show the compatibility between the actions of a Fourier--Mukai functor on Hochschild homology and singular cohomology which was conjecturally expected to hold true (see, for example, \cite{HFM,Markarian2}). Our result in this direction might be of independent interest and is the content of the following theorem, which will be proved in Section \ref{subsec:proofThm2}.

\begin{thm}\label{thm:main2}
    Let $X_1$ and $X_2$ be smooth complex projective varieties and let $\ke\in\Db(X_1\times X_2)$.
    Then the following diagram
    \begin{equation*}
        \xymatrix{\HH_*(X_1)\ar[rr]^{(\FM{\ke})_\HH}\ar[d]_{I_K^{X_1}}&&\HH_*(X_2)\ar[d]^{I_K^{X_2}}\\ H^*(X_1,\CC)\ar[rr]^{(\FM{\ke})_H}&&H^*(X_2,\CC)}
    \end{equation*}
    commutes.
\end{thm}

The isomorphisms $I_K^{X_i}$ in the diagram above are modifications of the classical Hochschild--Kostant--Rosenberg isomorphisms by the square root of the Todd class of $X_i$ (see \cite{Kont,Markarian2,Cal2}). A weaker version of Theorem \ref{thm:main2} is proved in \cite{Cal2} for algebraic classes.

The paper ends with a discussion about the applications to the cases of Kummer and Enriques surfaces (see Section \ref{sec:applications}).

\smallskip

In this paper we will always work over the complex numbers. For $X$ a smooth projective variety, $\Db(X):=\Db(\coh(X))$.

\section{Hochschild homology and singular cohomology}\label{sec:2}

This section is mainly devoted to the proof of the commutativity of the actions of Fourier--Mukai equivalences on Hochschild homology and singular cohomology. To this
end, we first show that such actions behave nicely with respect to the K\"{u}nneth decomposition. The proof is then carried out using directly Ramadoss' results \cite{Markarian2,Ramadoss}.

\subsection{Hochschild homology and cohomology}\label{subsec:Hoch}

Let $X$ be a smooth projective variety and denote by $\omega_X$ its dualizing sheaf. Given the diagonal embedding $\Delta_X:X\hookrightarrow X\times X$, we define $S_X:=\omega_X[\mathrm{dim}(X)]$, $S^{-1}_X:=\omega\dual_X[-\mathrm{dim}(X)]$ and $S^{\pm 1}_{\Delta_X}:=(\Delta_X)_*S^{\pm 1}_X$. The \emph{$i$-th Hochschild homology and cohomology groups}, $i\in\ZZ$, are now respectively (see, for example, \cite{Cal1})
\[
\begin{split}
\HH_i(X)&:=\Hom_{\Db(X\times X)}(S^{-1}_{\Delta_X}[i],\ko_{\Delta_X})\cong\Hom_{\Db(X)}(\ko_X[i],\mathbf{L}\Delta_X^*\ko_{\Delta_X})\\
\HH^i(X)&:=\Hom_{\Db(X\times X)}(\ko_{\Delta_X},\ko_{\Delta_X}[i])\cong\Hom_{\Db(X)}(\mathbf{L}\Delta_X^*\ko_{\Delta_X},\ko_X[i]).
\end{split}
\]
Hence one sets $\HH_*(X):=\bigoplus_i\HH_i(X)$ and $\HH^*(X):=\bigoplus_i\HH^i(X)$.

The \emph{Hochschild--Kostant--Rosenberg isomorphism}
$\mathbf{L}\Dd_X^*\ko_{\Dd_X}\isomor\bigoplus_i\Omega^i_X[i]$ (see
\cite{Cal2,Yek2}) yields the graded isomorphisms
\[
I^X_\mathrm{HKR}:\HH_*(X)\rightarrow\HO_*(X):=\bigoplus_i\HO_i(X)\quad\text{and}\quad I_X^\mathrm{HKR}:\HH^*(X)\rightarrow\HT^*(X):=\bigoplus_i\HT^i(X),
\]
where $\HO_i(X):=\bigoplus_{q-p=i}H^p(X,\Omega_X^q)$ and $\HT^i(X):=\bigoplus_{p+q=i}H^p(X,\wedge^q\kt_X)$. One then defines the graded isomorphisms
\[
I^X_K=(\mathrm{td}(X)^{1/2}\wedge(-))\circ I^X_\mathrm{HKR}\quad\text{and}\quad I_X^K=(\mathrm{td}(X)^{-1/2}\lrcorner(-))\circ I_X^\mathrm{HKR}.
\]

For $X_1$ and $X_2$ smooth and projective varieties and for
$\ke\in\Db(X_1\times X_2)$, let $\FM{\ke}:\Db(X_1)\to\Db(X_2)$ be
the \emph{Fourier--Mukai functor} with kernel $\ke$, i.e.\ the
functor $$\FM{\ke}:=\mathbf{R}(p_2)_*(\ke\lotimes p_1^*(-)),$$
where $p_i:X_1\times X_2\to X_i$ is the $i$-th projection. The
left and right adjoints of $\FM{\ke}$ are again Fourier--Mukai
functors with kernels $\ke_L:=\ke\dual\otimes p_2^*S_{X_2}$ and
$\ke_R:=\ke\dual\otimes p_1^*S_{X_1}$ respectively (see
\cite[Prop.\ 5.9]{HFM}). Notice that a Fourier--Mukai functor
extends to a functor
$\FM{\ke}:\mathrm{D}(\qcoh(X_1))\to\mathrm{D}(\qcoh(X_1))$.

From this, we get functorially a graded morphism $(\FM{\ke})_\HH:\HH_*(X_1)\to\HH_*(X_2)$ such that, if $\mu\in\HH_i(X_1)=\Hom(S^{-1}_{\Delta_{X_1}}[i],\ko_{\Delta_{X_1}})$, then $(\FM{\ke})_\HH(\mu)\in\HH_i(X_2)$ is defined (see \cite[Sect.\ 4.3]{Cal1}) by the composition
\begin{equation}\label{eqn:2}
S^{-1}_{\Delta_{X_2}}[i]\mor[\gamma]\ke\circ\ke\dual[i]\mor[\id\circ\eta\circ\id]\ke\circ S^{-1}_{\Delta_{X_1}}[i]\circ S_{\Delta_{X_1}}\circ\ke\dual\mor[\id\circ\mu\circ\id\circ\id]\ke\circ S_{\Delta_{X_1}}\circ\ke\dual\mor[\rote]\ko_{\Delta_{X_2}},
\end{equation}
where $\kf\circ\kg$ is the kernel of the composition
$\FM{\kf}\circ\FM{\kg}$ (see e.g.\ \cite[Prop.\ 5.10]{HFM}), $\eta:\ko_{\Delta_{X_1}}\to S^{-1}_{\Delta_{X_1}}\circ S_{\Delta_{X_1}}$ is the isomorphism coming from the easy fact that $\FM{S^{-1}_{\Delta_{X_1}}}\circ\FM{S_{\Delta_{X_1}}}=\id$ and
the morphisms $\gamma$ and $\rote$ are the natural ones (see
\cite[Appendix A]{Cal1}). If $\FM{\ke}$ is an equivalence, then
there exists also an action $(\FM{\ke})^\HH$ on Hochschild
cohomology induced by the functor
$\FM{\ke\boxtimes\kp}:\Db(X_1\times X_1)\to\Db(X_2\times X_2)$,
where $\kp\cong\ke_L\cong\ke_R$ is the kernel of the inverse of
$\FM{\ke}$, which sends $\ko_{\Dd_{X_1}}$ to $\ko_{\Dd_{X_2}}$
(see, for example, \cite[Remark 6.3]{HFM}). The following easy result is
probably well-known (see, for example, \cite[Sect.\ 9.4]{Weibel} for a similar statement about Hochschild homology of rings).

\begin{lem}\label{lem:Kunneth}
    Let $X_1$, $X_1'$, $X_2$ and $X'_2$ be smooth projective varieties.

{\rm (i)} There exists a natural isomorphism
    \begin{equation*}
        \kappa_{X_1,X_2}:\HH_*(X_1)\otimes\HH_*(X_2)\rightarrow\HH_*(X_1\times X_2),\qquad (\aa,\bb)\mapsto\aa\boxtimes\bb:=p_{X_1}^*\aa\otimes p_{X_2}^*\bb
    \end{equation*}
    respecting the functoriality of Hochshild homology and compatible with the Hochschild--Kostant--Rosenberg isomorphism, i.e.\ the diagram
    \[
        \xymatrix{
        \HH_*(X_1)\otimes\HH_*(X_2)\ar[rr]^{\kappa_{X_1,X_2}}\ar[d]_{I_\mathrm{HKR}^{X_1}\otimes I_\mathrm{HKR}^{X_2}}&&\HH_*(X_1\times X_2)\ar[d]^{I_\mathrm{HKR}^{X_1\times X_2}}\\
        \HO_*(X_1)\otimes\HO_*(X_2)\ar[rr]^{\wedge}&&\HO_*(X_1\times X_2)
        }
    \]
    commutes.

{\rm (ii)} Given two Fourier--Mukai functors $\FM{\ke}:\Db(X_1)\to\Db(X'_1)$ and $\FM{\kf}:\Db(X_2)\to\Db(X'_2)$, we have $$(\FM{\ke\boxtimes\kf})_\HH(\mu\boxtimes\nu)=(\FM{\ke})_\HH(\mu)\boxtimes(\FM{\kf})_\HH(\nu),$$ for any $\mu\in\HH_*(X_1)$ and $\nu\in\HH_*(X_2)$.
\end{lem}

\begin{proof}
Following \cite{Yek2}, for any smooth projective variety $X$, the complex $\mathbf{L}\Delta_X^*\ko_{\Delta_X}$ can be represented by the complex of $\ko_X$-modules $\widehat{C}^\bullet(X)$ such that, if $q>0$, then $\widehat{C}^{q}(X)=0$, while, for $q\leq 0$ and for any affine $U=\mathrm{Spec}(R)\subseteq X$, the group of sections $\Gamma(U,\widehat{C}^{q}(X))$ is an adic completion of the usual module of Hochschild chains $C^q(R)=R^{\otimes -q+2}\otimes_{R\otimes R}R$. The differential is induced by the differential $d^n:R^{\otimes n+1}\to R^{\otimes n}$ such that
    \[
    \begin{split}
        d^n(r_0\otimes\ldots\otimes r_n):=&r_0r_1\otimes r_2\otimes\ldots\otimes r_n-r_0\otimes r_1r_2\otimes\ldots\otimes r_n+\ldots\\
        &\ldots+(-1)^{n-1}r_0\otimes\ldots\otimes r_{n-1}r_n+(-1)^nr_0r_n\otimes\ldots\otimes r_{n-1}.
    \end{split}
    \]
Moreover, the isomorphism $I^X_\mathrm{HKR}$ is given, locally in $U$, by the morphism $I^n(r_0\otimes\ldots\otimes r_n)=\frac{1}{n!}r_0\mathrm{d}r_1\wedge\ldots\wedge\mathrm{d}r_n$ of $R$-modules.

For any pair of positive integers $p,q$, a $(p,q)$-shuffle is a permutation $\sigma$ of $\{1,\ldots,p+q\}$ such that $\sigma(1)<\ldots<\sigma(p)$ and $\sigma(p+1)<\ldots<\sigma(p+q)$. The morphism $\mathrm{sh}:R_1^{\otimes p+1}\otimes_\CC R_2^{\otimes q+1}\to (R_1\otimes_\CC R_2)^{\otimes p+q+1}$
\[
\mathrm{sh}((r_0\otimes\ldots\otimes r_p)\otimes_\CC(r'_0\otimes r_{p+1}\otimes\ldots\otimes r_{p+q})):=\sum_{\sigma\in\mathrm{Sh}(p,q)}(-1)^\sigma r_0r'_0\otimes r_{\sigma^{-1}(1)}\otimes\ldots\otimes r_{\sigma^{-1}(p+q)},
\]
where $\mathrm{Sh}(p,q)$ is the set of all $(p,q)$-shuffles and $X_i$ is locally $\mathrm{Spec}(R_i)$, gives the local description of the isomorphisms $\kappa_{X,Y}$. The compatibility with the Hochschild--Kostant--Rosenberg isomorphism in (i) is now an easy check based on the previous local descriptions. (As pointed out by the referee, (i) could be probably proved directly by using the adjunction between the Hochschild--Kostant--Rosenberg isomorphism and the universal Atiyah class \cite{Cal2}.)

To prove (ii), observe that $\FM{\ke\boxtimes\kf}(\kg_1\boxtimes\kg_2)=\FM{\ke}(\kg_1)\boxtimes\FM{\kf}(\kg_2)$ and a similar decomposition holds true for the right and left adjoints. Since $\ko_{\Delta_{X_1\times X_2}}\cong\ko_{\Delta_{X_1}}\boxtimes\ko_{\Delta_{X_2}}$ and $S^{-1}_{\Delta_{X_1\times X_2}}\cong S^{-1}_{\Delta_{X_1}}\boxtimes S^{-1}_{\Delta_{X_2}}$, the statement follows directly from the definition of the action of $\FM{\ke\boxtimes\kf}$ on Hochschild homology, since all the morphisms in \eqref{eqn:2} preserve the $\boxtimes$-product.
\end{proof}

\begin{remark}\label{rmk:HochCohom}
    In the case of Fourier--Mukai equivalences, an analogous result can be proved for Hochschild cohomology by applying the same proof.
\end{remark}

For later use, observe that any Fourier--Mukai functor $\FM{\ke}:\Db(X_1)\to\Db(X_2)$ also induces an action on singular cohomology $(\FM{\ke})_H:H^*(X_1,\CC)\to H^*(X_2,\CC)$ by $(\FM{\ke})_H(a):=(p_2)_*(v(\ke).p_1^*(a))$, where $v(\ke):=\ch(\ke).\sqrt{\mathrm{td}(X_1\times X_2)}$ is the \emph{Mukai vector} of $\ke$.

\subsection{Proof of Theorem \ref{thm:main2}}\label{subsec:proofThm2} For any smooth projective variety $X$, the Hochschild homology carries a non-degenerate pairing $\langle-,-\rangle_C:\HH_i(X)\times\HH_{-i}(X)\to\CC$ such that, according to \cite[Sect.\ 5]{Cal1}, for any $\mu\in\HH_i(X)$ and $\nu\in\HH_{-i}(X)$,
\[
\langle\mu,\nu\rangle_C:=\mathrm{tr}_{X\times X}(\tau(\mu)\circ\nu),
\]
where $\tau:\RHom_{X\times
X}(S^{-1}_{\Delta_X},\ko_{\Delta_X})\isomor\RHom_{X\times
X}(\ko_{\Delta_X},S_{\Delta_X})$ is the homomorphism obtained by
tensoring on the right by $p_2^*S_X$, with $p_2:X\times X\to X$
the second projection, and making the natural identifications.

If $\ke\in\Db(X_1\times X_2)$, then $\ke$ can be alternatively
seen as a complex in $\Db(\mathrm{pt}\times X_1\times X_2)$ and
hence there is a morphism
$(\FM{\ke})_\HH:\HH_0(\mathrm{pt})\to\HH_0(X_1\times X_2)$.
Following \cite{Markarian2,Cal1} one then defines the \emph{Chern character
of $\ke$} as the element
$\ch(\ke):=(\FM{\ke})_\HH(1)\in\HH_0(X_1\times X_2)$. The
comparison with the standard Chern character in singular
cohomology is the content of \cite[Thm.\ 4.5]{Cal2} which yields
$v(\ke)=I^{X_1\times X_2}_K(\ch(\ke))$.

With this in mind, one proves that the following equality holds true
    \begin{equation}\label{eqn:1}
    \langle\ch(\ke\dual),\mu\boxtimes\nu\rangle_C=\int_{X_2} I^{X_2}_K((\FM{\ke})_\HH(\mu))\wedge I^{X_2}_K(\nu),
    \end{equation}
    for any $\mu\in\HH_*(X_1)$ and $\nu\in\HH_*(X_2)$.

Indeed, first define the functor
    \begin{equation*}
        \Psi_\ke:\Db(X_1\times X_2)\to\Db(\mathrm{pt}),\qquad \kf\mapsto\RGamma(\ke\Lotimes\kf).
    \end{equation*}
An easy computation of the kernel of the composition of Fourier--Mukai functors shows that $\Psi_\ke$ is isomorphic to the following composition
    \[
    \Db(X_1\times X_2)\mor[\FM{\ke}\times\id]\Db(X_2\times X_2)\mor[\mathbf{L}\Delta_{X_2}^*]\Db(X_2)\mor[\RGamma]\Db(\mathrm{pt}),
    \]
where $\FM{\ke}\times\id:=\FM{\ke\boxtimes\ko_{\small{\Delta_{X_2}}}}$.

By definition $(\Psi_\ke)_\HH(\mu\boxtimes\nu)=\langle 1,(\Psi_\ke)_{\HH}(\mu\boxtimes\nu)\rangle_C$. By adjunction (\cite[Thm.\ 8]{Cal1}) and the definition of Chern character, the latter is equal to $\langle\ch(\ke\dual),\mu\boxtimes\nu\rangle_C$.
Applying adjunction once more and Lemma \ref{lem:Kunneth}(ii), we have
    \[
    (\Psi_\ke)_\HH(\mu\boxtimes\nu)=\langle\ch(\ko\dual_{\small{\Delta_{X_2}}}),(\FM{\ke}\times\id)_\HH(\mu\boxtimes\nu)\rangle_C=\langle\ch(\ko\dual_{\small{\Delta_{X_2}}}),(\FM{\ke})_\HH(\mu)\boxtimes\nu\rangle_C.
    \]
Following \cite{Ramadoss}, denote by $K:\HO_*(X_i)\to\HO_*(X_i)$
the involution such that $K|_{H^q(X_i,\Omega_{X_i}^p)}=(-1)^q$.
Using \cite[Prop.\ 3]{Ramadoss} (or more precisely equation (8) in
\cite{Ramadoss}) and Lemma \ref{lem:Kunneth}(i), we get the
following chain of equalities
    \[
    \begin{split}
    \langle\ch(\ko\dual_{\small{\Delta_{X_2}}}),(\FM{\ke})_\HH(\mu)\boxtimes\nu\rangle_C&=\int_{X_2\times X_2} K(\ch(\ko\dual_{\small{\Delta_{X_2}}}))\wedge I^{X_2\times X_2}_\mathrm{HKR}((\FM{\ke})_\HH(\mu)\boxtimes\nu)\wedge\mathrm{td}(X_2\times X_2)\\
    &=\int_{X_2\times X_2} \ch(\ko_{\small{\Delta_{X_2}}})\wedge I^{X_2}_\mathrm{HKR}((\FM{\ke})_\HH(\mu))\wedge I^{X_2}_\mathrm{HKR}(\nu)\wedge\mathrm{td}(X_2\times X_2)\\
    &=\int_{X_2} I^{X_2}_K((\FM{\ke})_\HH(\mu))\wedge I^{X_2}_K(\nu),
    \end{split}
    \]
for any $\mu\in\HH_*(X_1)$ and $\nu\in\HH_*(X_2)$. This proves \eqref{eqn:1}.

Arguing in the same way, one shows that the following identities are true
\[
\begin{split}
    \langle\ch(\ke\dual),\mu\boxtimes\nu\rangle_C&=\int_{X_1\times X_2} K(\ch(\ke\dual))\wedge I^{X_1}_\mathrm{HKR}(\mu)\wedge I^{X_2}_\mathrm{HKR}(\nu)\wedge\mathrm{td}(X_1\times X_2)\\
    &=\int_{X_1\times X_2} v(\ke)\wedge I^{X_1}_K(\mu)\wedge I^{X_2}_K(\nu)\\
    &=\int_{X_2} (\FM{\ke})_H(I^{X_1}_K(\mu))\wedge I^{X_2}_K(\nu),
\end{split}
\]
for any $\mu\in\HH_*(X_1)$ and $\nu\in\HH_*(X_2)$. By \eqref{eqn:1}, we have
\[
\int_{X_2} I^{X_2}_K((\FM{\ke})_\HH(\mu))\wedge I^{X_2}_K(\nu)=\int_{X_2} (\FM{\ke})_H(I^{X_1}_K(\mu))\wedge I^{X_2}_K(\nu),
\]
for any $\mu\in\HH_*(X_1)$ and $\nu\in\HH_*(X_2)$.

From the fact that the pairing $\int_{X_2}(-\wedge -)$ is
non-degenerate, $I^{X_2}_K\circ(\FM{\ke})_\HH=(\FM{\ke})_H\circ
I^{X_1}_K$. The theorem now follows from the natural
identification $H^*(X,\CC)\cong\HO_*(X)$ given by the Hodge
decomposition.

\begin{remark}\label{rmk:pairing}
    As observed in \cite{Ramadoss}, the singular cohomology groups $H^*(X_i,\CC)$ carry a non-degenerate pairing $\langle-,-\rangle_R:H^*(X_i,\CC)\times H^*(X_i,\CC)\to\CC$ given by the formula (implicit in the previous proof)
    \[
    \langle a,b\rangle_R:=\int_{X_i}\frac{K(a)}{\sqrt{\ch(\omega_{X_i})}}\wedge b,
    \]
    for any $a,b\in H^*(X_i,\CC)$. If $\ke\in\Db(X_1\times X_2)$ is the kernel of a Fourier--Mukai equivalence, then the commutativity of the diagram in Theorem \ref{thm:main2}  and \cite[Prop.\ 5]{Ramadoss} yields also the compatibility with the natural pairings, i.e.\
    \[
    \begin{split}
    \langle \alpha,\beta\rangle_C&=\langle I_K^{X_1}(\alpha),I_K^{X_1}(\beta)\rangle_R=\langle(\FM{\ke})_\HH(\alpha),(\FM{\ke})_\HH(\beta)\rangle_C\\&=\langle I_K^{X_2}((\FM{\ke})_\HH(\alpha)),I_K^{X_2}((\FM{\ke})_\HH(\beta))\rangle_R=\langle(\FM{\ke})_H(I_K^{X_1}(\alpha)),(\FM{\ke})_H(I_K^{X_1}(\beta))\rangle_R
    \end{split}
    \]
    for any $\alpha,\beta\in\HH_*(X_1)$.
    The \emph{generalized Mukai pairing} $\langle-,-\rangle_M:H^*(X_i,\CC)\times H^*(X_i,\CC)\to\CC$, introduced in \cite[Def.\ 3.2]{Cal2}, is the non-degenerate pairing such that, for any $a,b\in H^*(X_i,\CC)$,
\begin{equation}\label{eqn:Mukai}
    \langle a,b\rangle_M:=\langle\widetilde\tau(a),b\rangle_R
\end{equation}
    where $\widetilde\tau|_{H^q(X,\Omega_X^p)}=(\sqrt{-1})^{-p-q}$. Such a pairing is also compatible with the action of $\FM{\ke}$ on the singular cohomology.
\end{remark}

\section{Infinitesimal deformations}\label{sec:3}

In this section we prove Theorem \ref{thm:main1} which relates the existence of equivalences between some first order deformations of the derived categories of coherent sheaves and the existence of special Hodge isometries of the total cohomology. To this end, in Section \ref{subsec:category} we briefly recall the description of the first order deformations of $\coh(X)$ given in \cite{Toda}, for $X$ a smooth projective variety. Notice that an equivalent theory can be obtained using the general results in \cite{Lowen,LvdB1,LvdB2}. Although the first approach is the preferred one in this paper, the latter will also be made use of at some specific points.

After this, in Section \ref{subsec:proof2}, we introduce a special weight-$2$ decomposition of the total cohomology groups (tensored by $\ZZ[\epsilon]/(\epsilon^2)$) preserved by the action of Fourier--Mukai equivalences.

\subsection{The categorical setting}\label{subsec:category}

For $X$ a smooth projective variety and $v\in\HH^2(X)$, following \cite{Toda}, we consider the $\CC[\epsilon]/(\epsilon^2)$-linear abelian category $\coh(X,v)$ which is the first order deformation of $\coh(X)$ in the direction $v$. Since the precise definition of this category is not needed in the rest of this paper, we just recall the  essentials of its construction.

Write $I^\mathrm{HKR}_X(v)=(\alpha,\beta,\gamma)\in\HT^2(X)=H^2(X,\ko_X)\oplus H^1(X,\kt_X)\oplus H^0(X,\wedge^2\kt_X)$. Then one defines a sheaf $\ko_X^{(\beta,\gamma)}$ of $\CC[\epsilon]/(\epsilon^2)$-algebras on $X$ depending only on $\beta$ and $\gamma$.
Representing $\alpha\in H^2(X,\ko_X)$ as a \v{C}ech $2$-cocycle $\{\alpha_{ijk}\}$ one has an element $\widetilde\alpha:=\{1-\epsilon\alpha_{ijk}\}$ which is a \v{C}ech $2$-cocycle with values in the invertible elements of the center of $\ko^{(\beta,\gamma)}_X$.
In analogy with the classical twisted setting, we get the abelian category $\coh(\ko_X^{(\beta,\gamma)},\widetilde{\alpha})$ of $\widetilde\alpha$-twisted coherent $\ko^{(\beta,\gamma)}_X$-modules. Now set $\coh(X,v):=\coh(\ko_X^{(\beta,\gamma)},\widetilde{\alpha})$ and $\mathrm{D}^{*}(X,v):=\mathrm{D}^{*}(\coh(X,v))$, where $*=\mathrm{b},\pm,\emptyset$. Analogously, one defines the abelian category $\qcoh(X,v)$ of $\widetilde\alpha$-twisted quasi-coherent $\ko^{(\beta,\gamma)}_X$-modules, as the first order deformation of $\qcoh(X)$.

\begin{remark}\label{rmk:Lowen}
    The construction of the abelian categories sketched above is a geometric incarnation of a more abstract theory developed in \cite{LvdB2,LvdB1}. The connection between the two approaches can explained using \cite{Lowen}, where the flat deformations of the abelian category of (quasi-)coherent sheaves on a variety are shown to be equivalent to flat deformations of the structure sheaf of $X$ as a \emph{twisted presheaf} (see \cite[Thm.\ 1.4]{Lowen}). Such deformations are precisely the ones studied in \cite{Toda}.

    The obstruction theory for lifting objects in these deformations has been developed in \cite{Lowen2}. In particular, the obstruction class to deforming an object $\ke\in\Db(X)$ lives in $\Ext^2(\ke,\ke)$, while all possible deformations form an affine space over $\Ext^1(\ke,\ke)$.
\end{remark}

\begin{remark}\label{rmk:functors}
    According to \cite[Sect.\ 4]{Toda}, the usual derived functors are well-defined in this context, provided the correct compatibilities. For example, consider the obvious morphism of algebras $\iota^\sharp:\ko_X^{(\beta,\gamma)}\to\ko_X$ which, for simplicity, will be denoted by $\iota$. Then we have the functors $\iota_*:\Db(X)\to\Db(X,v)$, $\mathbf{L}\iota^*:\mathrm{D}^-(X,v)\to\mathrm{D}^-(X)$ and $\iota^!:\mathrm{D}^-(X,v)\to\mathrm{D}^-(X)$, where $\iota^!(-):=\bar\iota^*\mathrm{R}\kh om_{\mathrm{D}^-(X,v)}(\iota_*\ko_X,-)$ and $\bar\iota^*$ is induced by the natural exact functor between the categories of coherent $\iota_*\ko_X$-modules and $\coh(X)$. The functors $\mathbf{L}\iota^*$ and $\iota^!$ are, respectively, left and right adjoints of $\iota_*$.

Since there is no ambiguity, here as in the rest of the paper, we denote a functor $F:\Db(X)\to\Db(Y)$ and its (possible) extension to $\Db(X,v)\to\Db(Y,w)$ in the same way.
\end{remark}

As observed in the introduction, by the definition of $\coh(X,v)$, the notion of perfect complex makes sense also in this context. The category of perfect complexes is denoted by $\Dp(X,v)$. Take two smooth projective varieties $X_1$ and $X_2$, $v_i\in\HH^2(X_i)$ and $\widetilde\ke\in\Dp(X_1\times X_2,-J(v_1)\boxplus v_2)$, where $$-J(v_1)\boxplus v_2:=-p_1^*J(v_1)+p_2^*v_2$$ for $p_i:X_1\times X_2\to X_i$ the projection, and $J:\HH^2(X_1)\to\HH^2(X_1)$ is such that $(I_{X_1}^\mathrm{HKR}\circ J\circ (I_{X_1}^\mathrm{HKR})^{-1})(\alpha,\beta,\gamma)=(\alpha,-\beta,\gamma)$. (Notice that, when we write $-J(v_1)\boxplus v_2$, we are implicitly using Remark \ref{rmk:HochCohom}.) Then the Fourier--Mukai functor $$\FM{\widetilde\ke}:\Db(X_1,v_1)\to\Db(X_2,v_2)$$ is well-defined.

\begin{prop}\label{prop:FM}
    Let $X_1$ and $X_2$ be smooth projective varieties and $v_i\in\HH^2(X_i)$. Take $\widetilde\ke\in\Dp(X_1\times X_2,-J(v_1)\boxplus v_2)$ and set $\ke:=\mathbf{L}\iota^*\widetilde\ke$.

    {\rm (i)} $\FM{\widetilde\ke}\circ\iota_*\cong\iota_*\circ\FM{\ke}:\Db(X_1)\to\Db(X_2,v_2)$ and $\iota^!\circ\FM{\widetilde\ke}\cong\FM{\ke}\circ\iota^!:\Db(X_1,v_1)\to\mathrm{D}^-(X_2)$.

    {\rm (ii)} If $\FM{\widetilde\ke}:\Db(X_1,v_1)\to\Db(X_2,v_2)$ is an equivalence, then $\FM{\ke}:\Db(X_1)\to\Db(X_2)$ is an equivalence as well.
\end{prop}

\begin{proof}
The fact that $\FM{\widetilde\ke}\circ\iota_*\cong\iota_*\circ\FM{\ke}$ was already remarked in \cite[Thm.\ 4.7]{Toda} and it is an easy application of the projection formula and flat base change (which in this context hold true as remarked in \cite{Toda}). Indeed, for any $\kf\in\Db(X_1)$,
\[
\begin{split}
\FM{\widetilde\ke}(\iota_*(\kf))&=\mathbf{R}(p_2)_*(\widetilde\ke\lotimes p_1^*(\iota_*(\kf)))\cong\mathbf{R}(p_2)_*(\widetilde\ke\lotimes\iota_*(p_1^*(\kf)))\\
&\cong\mathbf{R}(p_2)_*\iota_*(\mathbf{L}\iota^*\widetilde\ke\lotimes p_1^*(\kf))\cong\iota_*\mathbf{R}(p_2)_*(\ke\lotimes p_1^*(\kf))=\iota_*(\FM{\ke}(\kf)).
\end{split}
\]
For the second one, we have that, for any $\kf\in\Db(X_1,v_1)$,
\[
\begin{split}
\iota^!(\FM{\widetilde\ke}(\kf))&=\iota^!\mathbf{R}(p_2)_*(\widetilde\ke\lotimes p_1^*(\kf))\\
&\cong\mathbf{R}(p_2)_*(\iota^!(\widetilde\ke\lotimes p_1^*(\kf)))\\
&\cong\mathbf{R}(p_2)_*(\ke\lotimes \iota^!p_1^*(\kf))\\
&\cong\mathbf{R}(p_2)_*(\ke\lotimes p_1^*(\iota^!\kf))\\
&=\FM{\ke}(\iota^!(\kf)),
\end{split}
\]
where we used, as before, flat base change, the projection formula and the fact that $\widetilde\ke$ is a perfect complex.

\medskip

To prove (ii), observe first that, as an easy consequence of
\cite[Cor.\ 3.3, Lemma 4.3]{Toda}, the category $\qcoh(X_i,v_i)$
has enough injectives. Moreover, $\iota^!\widetilde\ki$ is
injective, for all injective objects
$\widetilde\ki\in\qcoh(X_i,v_i)$, and the pull-backs of the
injective objects in $\qcoh(X_i,v_i)$ via $\iota^!$ (co)generate
$\qcoh(X_i)$ (actually all injective objects in $\qcoh(X_i)$ are
of the form $\iota^!\widetilde\ki$, for some injective
$\widetilde\ki$, by \cite[Cor.\ 6.15]{LvdB2}).

For $\kf\in\Db(X_1)$ and an injective $\widetilde\ki\in\Db(X_1,v_1)$, we have the following chain of isomorphisms below
\[
\begin{split}
\Hom(\kf,\iota^!\widetilde\ki)&\cong\Hom(\iota_*\kf,\widetilde\ki)\\&\cong\Hom(\FM{\widetilde\ke}(\iota_*\kf),\FM{\widetilde\ke}(\widetilde\ki))\\&\cong\Hom(\iota_*\FM{\ke}(\kf),\FM{\widetilde\ke}(\widetilde\ki))\\
&\cong\Hom(\FM{\ke}(\kf),\iota^!(\FM{\widetilde\ke}(\widetilde\ki)))\\&\cong\Hom(\FM{\ke}(\kf),\FM{\ke}(\iota^!\widetilde\ki)),
\end{split}
\]
where the first and forth isomorphism are obtained by adjunction, the second one is simply the action of $\FM{\widetilde\ke}$, while the third and the last one are consequences of (i).
It is easy to check that the composition of all these isomorphisms is the action of $\FM{\ke}$.
Since the objects $\iota^!\widetilde\ki$ (co)generate the category $\qcoh(X_1)$, the previous calculation shows that $\FM{\ke}$ is fully-faithful.

Let $F:\Db(X_2,v_2)\to\Db(X_1,v_1)$ be a quasi-inverse of $\FM{\widetilde\ke}$. By (i) and adjunction, we have $F\circ\iota_*\cong\iota_*\circ\FM{\ke_L}$. Thus, for any $\kf\in\Db(X_2)$ and any injective $\widetilde\ki\in\qcoh(X_2,v_2)$,
\[
\begin{split}
\Hom(\kf,\iota^!\widetilde\ki)&\cong\Hom(\iota_*\kf,\widetilde\ki)\\&\cong\Hom(F(\iota_*\kf),F(\widetilde\ki))\\
&\cong\Hom(\iota_*\FM{\ke_L}(\kf),F(\widetilde\ki))\\&\cong\Hom(\FM{\widetilde\ke}(\iota_*\FM{\ke_L}(\kf)),\widetilde\ki)\\
&\cong\Hom(\iota_*\FM{\ke}(\FM{\ke_L}(\kf)),\widetilde\ki)\\&\cong\Hom(\FM{\ke}(\FM{\ke_L}(\kf)),\iota^!\widetilde\ki).
\end{split}
\]
As before, the objects $\iota^!\widetilde\ki$ (co)generate the category $\qcoh(X_2)$. Therefore we have an isomorphism $\kf\cong\FM{\ke}(\FM{\ke_L}(\kf))$ which is what we wanted. An alternative proof can be obtained using Serre duality as explained in Appendix \ref{sec:Appendix}.
\end{proof}

The following result is the key ingredient in understanding the relation between first order deformations of varieties and deformations of kernels of Fourier--Mukai equivalences.

\begin{thm}\label{thm:Toda} {\bf (Toda)}
    Let $X_1$ and $X_2$ be smooth projective varieties and let $v_i\in\HH^2(X_i)$. Assume that there exists a Fourier--Mukai equivalence $\FM{\ke}:\Db(X_1)\isomor\Db(X_2)$ with kernel $\ke\in\Db(X_1\times X_2)$. Then there exists an object $\widetilde\ke\in\Dp(X_1\times X_2,-J(v_1)\boxplus v_2)$ giving rise to a Fourier--Mukai equivalence $\FM{\widetilde\ke}:\Db(X_1,v_1)\isomor\Db(X_2,v_2)$ and such that $\ke\cong\mathbf{L}\iota^*\widetilde\ke$ if and only if $(\FM{\ke})^\HH(v_1)=v_2$.
\end{thm}

\begin{proof}
The existence of $\widetilde\ke$, given a Fourier--Mukai equivalence $\FM{\ke}$ such that $(\FM{\ke})^\HH(v_1)=v_2$, is precisely \cite[Thm.\ 4.7]{Toda}.
For the proof of the other implication, we use an argument suggested to us by Y.\ Toda. Assume that $(\FM{\ke})^\HH(v_1)=v_3$. Let $\kp\in\Db(X_2\times X_1)$ be the kernel of a quasi-inverse of $\FM{\ke}$. By the first part of the theorem, there exists $\widetilde\kp\in\Dp(X_2\times X_1,-J(v_3)\boxplus v_1)$ giving rise to an equivalence $\FM{\widetilde\kp}:\Db(X_2,v_3)\to\Db(X_1,v_1)$. The composition $G:=\FM{\widetilde\ke}\circ\FM{\widetilde\kp}$ induces an equivalence $\coh(X_2,v_3)\to\coh(X_2,v_2)$ (as deformations of $\coh(X_2)$). Indeed, by Proposition \ref{prop:FM}(i), $G(\iota_*\coh(X_2))\subseteq\iota_*\coh(X_2)$. Since the abelian categories $\coh(X_2,v_2)$ and $\coh(X_2,v_3)$ are generated by $\iota_*\coh(X_2)$ by extensions, $G$ yields the desired equivalence. But now, by \cite[Thm.\ 3.1]{LvdB1}, all the deformations of $\coh(X_2)$ are parametrized by $\HH^2(X_2)$. Thus $v_2=v_3$.
\end{proof}

\subsection{Infinitesimal Mukai lattices and the proof of Theorem \ref{thm:main1}}\label{subsec:proof2}

Let $X$ be a K3 surface and let $v\in\HH^2(X)$. Let $w:=I_K^{X}(\sigma_X)+\epsilon I_K^{X}(v\circ \sigma_X)\in\Htil(X,\ZZ)\otimes\CC[\epsilon]/(\epsilon^2)$, where $\sigma_X$ is a generator for $\HH_{2}(X)$ as a $\CC$-vector space. Here $\Htil(X,\ZZ)$ is the \emph{Mukai lattice} of $X$, i.e.\ the $\ZZ$-module $H^*(X,\ZZ)$ endowed with the (generalized) Mukai pairing and the weight-$2$ Hodge structure
\[
\Htil^{2,0}(X):=H^{2,0}(X)\qquad\Htil^{0,2}(X):=\overline{\Htil^{2,0}(X)}\qquad\Htil^{1,1}(X):=(\Htil^{2,0}(X)\oplus\Htil^{0,2}(X))^\perp.
\]

\begin{definition}\label{def:Mukailattice}
    The free $\ZZ[\epsilon]/(\epsilon^2)$-module of finite rank $\Htil(X,\ZZ)\otimes\ZZ[\epsilon]/(\epsilon^2)$ endowed with the $\ZZ[\epsilon]/(\epsilon^2)$-linear extension of the generalized Mukai pairing $\langle-,-\rangle_M$ and such that $\Htil(X,\ZZ)\otimes\CC[\epsilon]/(\epsilon^2)$
    has the weight-$2$ decomposition
    \[
    \begin{split}
    &\Htil^{2,0}(X,v):=\CC[\epsilon]/(\epsilon^2)\cdot w\\
    &\Htil^{0,2}(X,v):=\overline{\Htil^{2,0}(X,v)}\\
    &\Htil^{1,1}(X,v):=(\Htil^{2,0}(X,v)\oplus\Htil^{0,2}(X,v))^\perp,
    \end{split}
    \]
    is the \emph{infinitesimal Mukai lattice of $X$ with respect to $v$}, which is denoted by $\Htil(X,v,\ZZ)$.
\end{definition}

\begin{remark}\label{rmk:period}
(i) It is easy to see that $w$ behaves like a honest period of a K3 surface. More precisely,
\[
\langle w,w\rangle_M=0\qquad\langle w,\overline w\rangle_M>0,
\]
where we do not distinguish between the Mukai pairing \eqref{eqn:Mukai} and its $\ZZ[\epsilon]/(\epsilon^2)$-linear extension. Hence the weight-$2$ decomposition in the previous definition can be thought of as the analogue of the weight-$2$ Hodge decomposition on $\Htil(X,\ZZ)$ appearing in the classical Derived Torelli Theorem (see, for example, \cite{HFM} for the classical case and \cite{HS} for the twisted setting).

(ii) Going back to the motivation described in the introduction, suppose the morphsim $p:\mathrm{DF}_1\to\mathrm{DF}_2$ of functors is given. Theorem \ref{thm:main1} would then state the equivariance of the differential of $p$ with respect to the action of Fourier--Mukai equivalences on the first order versions of the functors $\mathrm{DF}_1$ and $\mathrm{DF}_2$.
\end{remark}

\medskip

Given this analogy, for two K3 surfaces $X_1$ and $X_2$, and $v_i\in\HH^2(X_i)$, a \emph{Hodge isometry} of the infinitesimal Mukai lattices is a $\ZZ[\epsilon]/(\epsilon^2)$-linear isomorphism $g:\Htil(X_1,v_1,\ZZ)\isomor\Htil(X_2,v_2,\ZZ)$ preserving the Mukai pairing and the weight-$2$ decomposition in the previous definition. In the rest of this paper, we will be more interested in the infinitesimal isometries $g=g_0\otimes\ZZ[\epsilon]/(\epsilon^2)$, where $g_0$ is (automatically) an Hodge isometry of the standard Mukai lattices $\Htil(X_1,\ZZ)\isomor\Htil(X_2,\ZZ)$. These special infinitesimal Hodge isometries will be called \emph{effective}.

\medskip

The lattice $\widetilde H(X_i,\ZZ)$ has some interesting
substructures. Indeed, let $\sigma_i$ be a generator of $H^{2,0}(X_i)$
and $\omega_i$ a K{\"a}hler class. Then
\begin{eqnarray}\index{Positive four-space}
P(X_i,\sigma_i,\omega_i):=\langle{\rm Re}(\sigma_i), {\rm
Im}(\sigma_i),1-\omega_i^2/2,\omega_i\rangle,
\end{eqnarray}
is a positive four-space in $\widetilde H(X_i,\RR)$ (here ${\rm
Re}(\sigma_i)$ and ${\rm Im}(\sigma_i)$ are the real and imaginary
part of $\sigma_i$). It comes, by the choice of basis, with a
natural orientation. An effective Hodge isometry $g=g_0\otimes\ZZ[\epsilon]/(\epsilon^2):\Htil(X_1,v_1,\ZZ)\isomor\Htil(X_2,v_2,\ZZ)$ is
\emph{orientation preserving} if $g_0$ preserves the orientation
of $P(X,\sigma_i,\omega_i)$, for $i=1,2$. For $X=X_1=X_2$ and
$v=v_1=v_2$, the group of orientation preserving effective Hodge
isometries is denoted by $\OO_+(\Htil(X,v,\ZZ))$.

\bigskip

\noindent {\it Proof of Theorem \ref{thm:main1}.} We first prove that (i) implies (ii). Let $\FM{\widetilde\ke}:\Db(X_1,v_1)\isomor\Db(X_2,v_2)$ be an equivalence with kernel $\widetilde\ke\in\Dp(X_1\times X_2,-J(v_1)\boxplus v_2)$. By Proposition \ref{prop:FM}(ii), $\ke:=\mathbf{L}\iota^*\widetilde\ke\in\Db(X_1\times X_2)$ is the kernel of a Fourier--Mukai equivalence $\FM{\ke}:\Db(X_1)\isomor\Db(X_2)$. Corollary 7.9 in \cite{HMS1} implies that the Hodge isometry $g_0:=(\FM{\ke})_H:\Htil(X_1,\ZZ)\to\Htil(X_2,\ZZ)$ is orientation preserving. Moreover, by Theorem \ref{thm:Toda}, since $\widetilde\ke$ is a first order deformation of $\ke$, $(\FM{\ke})^\HH(v_1)=v_2$.

Consider the diagram
\begin{equation}\label{eqn:diagrammabello}
    \xymatrix{
    \HH^*(X_1)\ar[rr]^{(\FM{\ke})^\HH}\ar[d]_{(-)\circ\sigma_{X_1}}&&\HH^*(X_2)\ar[d]^{(-)\circ(\FM{\ke})_\HH(\sigma_{X_1})}\\
    \HH_*(X_1)\ar[rr]^{(\FM{\ke})_\HH}\ar[d]_{I_K^{X_1}}&&\HH_*(X_2)\ar[d]^{I_K^{X_2}}\\
    \Htil(X_1,\CC)\ar[rr]^{(\FM{\ke})_H}&&\Htil(X_2,\CC),
    }
\end{equation}
where, as before, $\sigma_{X_1}$ is a generator of $\HH_2(X_1)$ as a $\CC$-vector space and $(-)\circ(-):\HH^*(X_i)\times\HH_*(X_i)\to\HH_*(X_i)$ denotes the action of $\HH^*(X_i)$ on $\HH_*(X_i)$ (see, for example, \cite{Cal1}). The upper square is commutative by, for example, \cite[Remark 6.3]{HFM}, while the commutativity of the bottom one is Theorem \ref{thm:main2}.
Thus
\[
g_0(I_K^{X_1}(v_1\circ\sigma_{X_1}))=I_K^{X_2}(v_2\circ((I_K^{X_1})^{-1}\circ g_0\circ I_K^{X_1})(\sigma_{X_1})).
\]
In particular, $g:=g_0\otimes
\ZZ[\epsilon]/(\epsilon^2):\Htil(X_1,v_1,\ZZ)\to\Htil(X_2,v_2,\ZZ)$
is an effective orientation preserving Hodge isometry of the
infinitesimal Mukai lattices.

The fact that (ii) implies (i) is shown as follows. Let $g=g_0\otimes\ZZ[\epsilon]/(\epsilon^2)$ be as in (ii), with $g_0:\Htil(X_1,\ZZ)\to\Htil(X_2,\ZZ)$ an orientation preserving Hodge isometry. By the classical Derived Torelli Theorem \cite{Mu,Or1}, there exists a Fourier--Mukai equivalence $\FM{\ke}:\Db(X_1)\isomor\Db(X_2)$ with kernel $\ke\in\Db(X_1\times X_2)$ and such that $g_0=(\FM{\ke})_H$.

Under our assumptions, the commutativity of diagram \eqref{eqn:diagrammabello} gives $(\FM{\ke})^\HH(v_1)=v_2$. Therefore, by Theorem \ref{thm:Toda}, there exists a first order deformation $\widetilde\ke\in\Dp(X_1\times X_2,-J(v_1)\boxplus v_2)$ of $\ke$ such that the Fourier--Mukai functor $\FM{\widetilde\ke}:\Db(X_1,v_1)\isomor\Db(X_2,v_2)$ is an equivalence.\hfill$\Box$

\bigskip

Theorem \ref{thm:main1} is precisely the classical result by Mukai
and Orlov \cite{Mu,Or1} if $v_1$ and $v_2$ in the statement are
trivial. In particular, under this assumption, condition (ii) can
be relaxed avoiding the orientation preserving requirement. This
is no longer true when $v_1$ and $v_2$ are non-trivial, as
explained in the example below.

\begin{ex}\label{ex:orientpres}
Let $X$ be a K3 surface with $\Pic(X)=\ZZ H$ and $H^2>2$. Take $v\in\HH^2(X)$ such that $w:=I_K^X(v\circ\sigma_X)=(\sqrt{2},\sqrt{3}H,\sqrt{5})\in\Htil^{1,1}(X)$, where $\sigma_X$ is a generator of $\HH_{2}(X)$. The Hodge isometry $j:=\id_{(H^0\oplus H^4)(X,\ZZ)}\oplus(-\id_{H^2(X,\ZZ)})$ does not preserve the orientation and maps $w$ to $w':=(\sqrt{2},-\sqrt{3}H,\sqrt{5})$. Let $v'\in\HH^2(X)$ be such that $w'=I_K^X(v'\circ\sigma_X)$. Then $\Db(X,v)$ is not Fourier--Mukai equivalent to $\Db(X,v')$.

Indeed, suppose for a contradiction that they are Fourier--Mukai equivalent. By Theorem \ref{thm:main1}, there is an orientation preserving Hodge isometry $g$ of $\Htil(X,\ZZ)$ such that $g(v)=\alpha v'$, for some root of unity $\aa$. But
\[
g(\sqrt{2},\sqrt{3}H,\sqrt{5})=(a_1\sqrt{2}+a_2\sqrt{3}+a_3\sqrt{5},(b_1\sqrt{2}+b_2\sqrt{3}+b_3\sqrt{5})H,c_1\sqrt{2}+c_2\sqrt{3}+c_3\sqrt{5}),
\]
with $a_i,b_i,c_i\in\ZZ$, $i=1,2,3$. An easy computation shows that $g(v)=\alpha v'$ only if $g$ restricted to $H^0(X,\ZZ)\oplus\Pic(X)\oplus H^4(X,\ZZ)$ is such that $g(x,yH,z)=\pm(x,-yH,z)$. By \cite[Lemma 4.1]{Og}, $g|_{T(X)}=\pm\id$ and hence by \cite[Thm.\ 1.6.1, Cor.\ 1.5.2]{Ni}, $g=\pm j$, which is a contradiction.

Notice that in such a case, no object in $\Db(X)$ deforms to an object in $\Db(X,v)$, or to one in $\Db(X,v')$.
\end{ex}

\begin{remark}\label{rmk:FMpartners}
    Given a K3 surface $X$, the number of isomorphism classes of K3 surfaces with derived category equivalent to $\Db(X)$ is finite (see \cite{BM}). The same result holds true in the broader case of twisted K3 surfaces (see \cite[Cor.\ 4.6]{HS}).

    As first order deformations of $X$ are parametrized by an affine space over $\HH^2(X)$ this cannot be true in the deformed setting. Hence one could weaken the notion of isomorphism between deformed K3 surfaces following \cite{HS}, requiring that $(X_1,v_1)$ and $(X_2,v_2)$ (where $X_i$ is a K3 surface and $v_i\in\HH^2(X_i)$) are \emph{equivalent deformations} if there exists an isomorphism $f:X_1\to X_2$ such that $f^*v_2=v_1$. Unfortunately, the number of Fourier--Mukai partners of a deformed K3 surface $(X,v)$ remains infinite even for this equivalence relation.

    Indeed, let $X$ be a K3 surface containing infinitely many smooth rational curves $\{C_i\}_{i\in\NN}$ and take $v\in\HH^2(X)$ such that $w:=I_K^X(v\circ\sigma_X)=(1,H,1)\in\Htil^{1,1}(X)$, for $H\in\Pic(X)$ an ample line bundle and $\sigma_X$ a generator of $\HH_2(X)$. If $s_i$ is the Hodge isometry of the total cohomology group of $X$ which acts as the reflection in the class of the rational curve $C_i$, then for any $r\in\NN$ $w_r:=(s_r\circ\ldots\circ s_1)(w)$ yields pairwise non-equivalent deformations $(X,v_r)$, with $v_r\in\HH^2(X)$, such that $\Db(X,v)\cong\Db(X,v_r)$ (by Theorem \ref{thm:main1}). The same example shows that the number of Fourier--Mukai partners is infinite, even if we declare the deformations $(X_1,v_1)$ and $(X_2,v_2)$ (with $X_i$ and $v_i\in\HH^2(X_i)$ as before) to be isomorphic if there exists an
equivalence $(L\otimes(-))\circ f^*:\coh(X_2,v_2)\isomor\coh(X_1,v_1)$, for some
isomorphism $f:X_1\isomor X_2$ and some line bundle $L\in\Pic(X_1)$.
\end{remark}

\subsection{Fourier--Mukai functors and the group of autoequivalences}\label{subsec:auto}

In this section we generalize a few properties of Fourier--Mukai functors to the case of deformed categories with the aim of specializing Theorem \ref{thm:main1} to autoequivalences.

Take $X_1$, $X_2$ and $X_3$ smooth projective varieties and $v_i\in\HH^2(X_i)$, where $i=1,2,3$ and $I^\mathrm{HKR}_{X_i}(v_i)=(\alpha_i,\beta_i,\gamma_i)$. Let $\widetilde\ke\in\Dp(X_1\times X_2,-J(v_1)\boxplus v_2)$ and $\widetilde\kf\in\Dp(X_2\times X_3,-J(v_2)\boxplus v_3)$. Define
\[
\widetilde\kf\circ\widetilde\ke:=\mathbf{R}(p_{13})_*(p_{12}^*\widetilde\ke\lotimes p_{23}^*\widetilde\kf)\in\Db(X_1\times X_3,-J(v_1)\boxplus v_3),
\]
where $p_{ij}$ are the natural projections from $X_1\times X_2\times X_3$. To unravel the definition, observe that, for $\{i,j\}\in\{\{1,2\},\{2,3\}\}$, one has
$p_{ij}^*:\Db(X_i\times X_j,-J(v_i)\boxplus v_j)\to\Db(X_1\times X_2\times X_3,w_{ij})$,
where
\[
\begin{split}
I^\mathrm{HKR}_{X_1\times X_2\times X_3}(w_{12})&=(p_{12}^*(-\alpha_1\boxplus\alpha_2),\beta_1\boxplus\beta_2\boxplus\beta_3,-\gamma_1\boxplus\gamma_2\boxplus 0)\\
I^\mathrm{HKR}_{X_1\times X_2\times X_3}(w_{23})&=(p_{23}^*(-\alpha_2\boxplus\alpha_3),\beta_1\boxplus\beta_2\boxplus\beta_3,0\boxplus-\gamma_2\boxplus\gamma_3).
\end{split}
\]
Moreover, we can tensor $p_{12}^*\widetilde\ke$ and $p_{23}^*\widetilde\kf$, seen respectively as objects in the derived categories of $q_2^{-1}\ko_{X_2}^{(\beta_2,\gamma_2)}$-modules and $q_2^{-1}\ko_{X_2}^{(\beta_2,-\gamma_2)}$-modules, where $q_2:X_1\times X_2\times X_3\to X_2$ is the projection. Such a tensor product takes naturally values in the derived category of $p_{13}^{-1}\ko_{X_1\times X_3}^{(\beta_1\boxplus\beta_3,-\gamma_1\boxplus\gamma_3)}$-modules. Hence, by \cite{Toda}, we can apply the functor $\mathbf{R}(p_{13})_*$. In this argument we did not take care of the twist because it behaves nicely with respect to the various operations.

\begin{lem}\label{lem:Kuz}
    Under the above assumptions, $\widetilde\kg:=\widetilde\kf\circ\widetilde\ke\in\Dp(X_1\times X_3,-J(v_1)\boxplus v_3)$.
\end{lem}

\begin{proof}
Let $\ke:=\mathbf{L}\iota^*\widetilde\ke$ and $\kf:=\mathbf{L}\iota^*\widetilde\kf$. To prove that $\widetilde\kg$ is perfect it is sufficient to show that $\mathbf{L}\iota^*\widetilde\kg$ is bounded. This is an easy consequence of the following isomorphisms
\[
\begin{split}
    \mathbf{L}\iota^*\widetilde\kg&=\mathbf{L}\iota^*\mathbf{R}(p_{13})_*(p_{12}^*\widetilde\ke\lotimes p_{23}^*\widetilde\kf)\\
    &\cong\mathbf{R}(p_{13})_*\mathbf{L}\iota^*(p_{12}^*\widetilde\ke\lotimes p_{23}^*\widetilde\kf)\\
    &\cong\mathbf{R}(p_{13})_*\mathbf{L}(p_{12}^*\ke\lotimes p_{23}^*\kf)\\
    &=\kf\circ\ke,
\end{split}
\]
where the natural isomorphism $\mathbf{L}\iota^*\mathbf{R}(p_{13})_*\cong\mathbf{R}(p_{13})_*\mathbf{L}\iota^*$ was already observed in the proof of Lemma 6.5 in \cite{Toda}. Such an isomorphism will be further clarified in Appendix \ref{sec:Appendix} (see Lemma \ref{lemma:kuz}).
\end{proof}

Due to this result, the composition of the Fourier--Mukai functors $\FM{\widetilde\kf}$ and $\FM{\widetilde\ke}$ is again a Fourier--Mukai functor with kernel $\widetilde\kf\circ\widetilde\ke$. In the case of non-deformed derived categories, the inverse of any Fourier--Mukai equivalence is a again of Fourier--Mukai type. This fact needs to be proved in the first order deformation case.

\begin{prop}\label{prop:inverse}
    If $X_1$ and $X_2$ are smooth projective varieties, $v_i\in\HH^2(X_i)$ and $$\FM{\widetilde\ke}:\Db(X_1,v_1)\isomor\Db(X_2,v_2)$$ is a Fourier--Mukai equivalence with $\widetilde\ke\in\Dp(X_1\times X_2,-J(v_1)\boxplus v_2)$, then the inverse is a Fourier--Mukai functor.
\end{prop}

\begin{proof}
This result can be easily proved using Serre duality as in Appendix \ref{sec:Appendix} (see Corollary \ref{cor:twisted}). Nevertheless we can also argue as follows. Observe first that, as an easy application of the projection formula, the identity $\Db(X_i,v_i)\to\Db(X_i,v_i)$ is a Fourier--Mukai functor whose kernel is $(\Delta_{X_i})_*\ko_{X_i}^{(\beta_i,\gamma_i)}$, where $I^{\mathrm{HKR}}_{X_i}(v_i)=(\alpha_i,\beta_i,\gamma_i)$. Notice that $\mathbf{L}\iota^*(\Delta_{X_i})_*\ko_{X_i}^{(\beta_i,\gamma_i)}\cong\ko_{\Delta_{X_i}}$.

Let $\kp\in\Db(X_2\times X_1)$ be the kernel of the inverse of the equivalence $\FM{\ke}$ (see Proposition \ref{prop:FM}(ii)), where $\ke:=\mathbf{L}\iota^*\widetilde\ke$. By Theorem \ref{thm:Toda}, there exists at least a $\widetilde\kp\in\Dp(X_2\times X_1,-J(v_2)\boxplus v_1)$ such that $\kp\cong\mathbf{L}\iota^*\widetilde\kp$. By \cite{Lowen2}, all such kernels are parametrized by $\Ext^1(\kp,\kp)$.

The functor $\widetilde\ke\circ(-):\Dp(X_2\times X_1,-J(v_2)\boxplus v_1)\to\Dp(X_2\times X_2,-J(v_2)\boxplus v_2)$ induces an isomorphism $\Ext^1(\kp,\kp)\isomor\Ext^1(\ko_{\Delta_{X_2}},\ko_{\Delta_{X_2}})$ and hence a one-to-one correspondence between deformations of $\kp$ and $\ko_{\Delta_{X_2}}$.
Therefore, by the previous computation, there exists a $\widetilde\kp$ such that $\widetilde\ke\circ\widetilde\kp\cong(\Delta_{X_2})_*\ko_{X_2}^{(\beta_2,\gamma_2)}$.
\end{proof}

As a consequence, for a smooth projective variety $X$ and $v\in\HH^2(X)$, the set $\Aut^\mathrm{FM}(\Db(X,v))$ of all autoequivalences of Fourier--Mukai type of $\Db(X,v)$ is actually a group.

As remarked in the introduction, when $X$ is a K3 surface, Theorem \ref{thm:main1} can be read in terms of the existence of a surjective group homorphism
\[
\xymatrix{\Pi_{(X,v)}:\Aut^\mathrm{FM}(\Db(X,v))\ar@{->>}[r]&\OO_+(\Htil(X,v,\ZZ))},
\]
where $\OO_+(\Htil(X,v,\ZZ))$ denotes the group of orientation preserving effective Hodge isometries. If $v=0$, the kernel of $\Pi_{(X,0)}$ consists of all autoequivalences acting trivially on cohomology and, according to Conjecture 1.2 in \cite{B}, should be described as the fundamental group of a period domain naturally associated to a connected component of the manifold parametrizing stability conditions on $\Db(X)$.

The relation with the case $v\neq 0$ is clarified by the following easy result.

\begin{lem}\label{lem:kernel}
    If $X$ is a K3 surface and $v\in\HH^2(X)$, then $\ker(\Pi_{(X,v)})\cong\ker(\Pi_{(X,0)})$.
\end{lem}

\begin{proof}
Any autoequivalence $\FM{\widetilde\ke}$ is in $\ker(\Pi_{(X,v)})$ if and only if $\Pi_{(X,v)}(\FM{\widetilde\ke})=\id\otimes\ZZ[\epsilon]/(\epsilon^2)$. In this case, by Proposition \ref{prop:FM}(ii) $\ke:=\mathbf{L}\iota^*\widetilde\ke$ is the kernel of a Fourier--Mukai equivalence which, by the proof of Theorem \ref{thm:main1}, acts trivially on cohomology. Hence, there exists a morphism $\kappa:\ker(\Pi_{(X,v)})\to\ker(\Pi_{(X,0)})$ sending $\widetilde\ke$ to $\ke$, which is surjective by Theorem \ref{thm:Toda}. By \cite{Lowen2}, given $\ke\in\Db(X\times X)$, all the $\widetilde\ke\in\Dp(X\times X, -J(v)\boxplus v)$ such that $\ke=\mathbf{L}\iota^*\widetilde\ke$ form an affine space over $\Ext^1(\ke,\ke)$ which, in the case of K3 surfaces, is trivial. Thus $\kappa$ is an isomorphism.
\end{proof}

\section{Further examples}\label{sec:applications}

Let $X$ be a smooth projective variety with an action of a finite
group $G$. We denote by $\coh_G(X)$ the abelian category of
$G$-equivariant coherent sheaves on $X$, i.e.\ the category whose
objects are pairs $(\ke,\{\lambda_g\}_{g\in G})$, where
$\ke\in\coh(X)$ and, for any $g_1,g_2\in G$,
$\lambda_{g_i}:\ke\isomor g_i^*\ke$ is an isomorphism such that
$\lambda_{g_1g_2}=g_2^*(\lambda_{g_1})\circ\lambda_{g_2}$. The set
of these isomorphisms is a \emph{$G$-linearization} of $\ke$ (very
often a $G$-linearization will be simply denoted by $\lambda$).
The morphisms in $\coh_G(X)$ are just the morphisms of coherent
sheaves compatible with the $G$-linearizations. We put
$\Db_G(X):=\Db(\coh_G(X))$. Since $G$ is finite, $\Db_G(X)$ can
equivalently be described in terms of $G$-equivariant objects in
$\Db(X)$ (see, for example, \cite[Sect.\ 1.1]{Pl}). For later use,
we recall the definition of the functor $\Inf_G:\Db(X)\to\Db_G(X)$
\[
\Inf_G(\ke):=\left(\bigoplus_{g\in G}g^*\ke,\lambda_\mathrm{nat}\right),
\]
where $\lambda_\mathrm{nat}$ is the natural $G$-linearization.

\subsection{Kummer surfaces}\label{subsec:Kummer} Let now $A$ be an abelian surface and denote by $\Km(A)$ the corresponding Kummer surface, i.e.\ the minimal resolution of the quotient of $A$ by the natural involution $\varpi:A\to A$, with $\varpi(a)=-a$. Denote by $G\cong\ZZ/2\ZZ$ the group generated by $\varpi$. The main result in \cite{BKR} shows that there exists a Fourier--Mukai equivalence $\Psi_A:\Db_G(A)\isomor\Db(\Km(A))$. The composition $\Pi_A:=\Psi_A\circ\Inf_G:\Db(A)\to\Db(\Km(A))$ is of Fourier--Mukai type and induces a morphism
\[
(\Pi_A)_\HH:\HH_*(A)\to\HH_*(\Km(A)).
\]
Thus, given $v\in\HH^2(A)$, we get $\pi(v,\sigma_A):=(\Pi_A)_\HH(v\circ\sigma_A)\circ(\Pi_A)_\HH(\sigma_A)^{-1}\in\HH^2(\Km(A))$, where $\sigma_A$ is a generator of $\HH_2(A)$.

Just as the classical Derived Torelli Theorem does, the infinitesimal version in Theorem \ref{thm:main1} holds true for abelian surfaces as well. From this we deduce a relation between the deformations of Fourier--Mukai equivalences in the case of abelian surfaces and the ones for the corresponding Kummer surfaces.

\begin{prop}\label{prop:appl1}
    Let $A_1$ and $A_2$ be abelian surfaces and let $v_i\in\HH^2(A_i)$, with $i=1,2$, be such that there exists a Fourier--Mukai equivalence $$\FM{\widetilde\ke}:\Db(A_1,v_1)\isomor\Db(A_2,v_2).$$ Then $\Db(\Km(A_1),\pi(v_1,\sigma_{A_1}))$ and $\Db(\Km(A_2),\pi(v_2,(\FM{\ke})_\HH(\sigma_{A_1})))$ are Fourier--Mukai equivalent, where $\ke:=\mathbf{L}\iota^*\widetilde\ke\in\Db(A_1\times A_2)$.
\end{prop}

\begin{proof}
The same argument as in \cite[Prop.\ 3.4]{MMS} shows that there exists $\kg\in\Db(A_1\times A_2)$ giving rise to an equivalence $\FM{\kg}:\Db(A_1)\isomor\Db(A_2)$ such that $(\varpi\times\varpi)^*\kg\cong\kg$ and $(\FM{\kg})_H=(\FM{\ke})_H$. Hence, by Theorem \ref{thm:main2}, the fact that $(\FM{\kg})_\HH=(\FM{\ke})_\HH$ and the commutativity of the diagram
\[
\xymatrix{
    \HH^*(A_1)\ar[rr]^{(\FM{\kg})^\HH}\ar[d]_{(-)\circ\sigma_{A_1}}&&\HH^*(A_2)\ar[d]^{(-)\circ(\FM{\kg})_\HH(\sigma_{A_1})}\\
    \HH_*(A_1)\ar[rr]^{(\FM{\kg})_\HH}&&\HH_*(A_2),
    }
\]
we have $(\FM{\kg})^\HH(v_1)=v_2$.

By the discussion in \cite[Sect. 3.2]{Pl} and Theorem \ref{thm:main2}, there exists an object $\kf\in\Db(\Km(A_1)\times\Km(A_2))$ inducing an equivalence $\FM{\kf}:\Db(\Km(A_1))\isomor\Db(\Km(A_2))$ and making commutative the following diagram
\[
\xymatrix{\HH_*(A_1)\ar[rrr]^{(\FM{\kg})_\HH}\ar[ddd]_{(\Pi_{A_1})_\HH}\ar[dr]_{I_K^{A_1}}&&&\HH_*(A_2)\ar[ddd]^{(\Pi_{A_2})_\HH}\ar[dl]^{I_K^{A_2}}\\
&\Htil(A_1,\CC)\ar[d]\ar[r]^{(\FM{\kg})_H}&\Htil(A_2,\CC)\ar[d]&\\
&\Htil(\Km(A_1),\CC)\ar[r]^{(\FM{\kf})_H}&\Htil(\Km(A_2),\CC)&\\
\HH_*(\Km(A_1))\ar[rrr]^{(\FM{\kf})_\HH}\ar[ur]^{I_K^{\Km(A_1)}}&&&\HH_*(\Km(A_2))\ar[ul]_{I_K^{\Km(A_2)}}.
}
\]
The commutativity of \eqref{eqn:diagrammabello}, for $X_i=\Km(A_i)$, and Theorem \ref{thm:Toda} yield the following chain of equalities
\[
\begin{split}
(\FM{\kf})^\HH(\pi(v_1,\sigma_{A_1}))&=(\FM{\kf})^\HH((\Pi_{A_1})_\HH(v_1\circ\sigma_{A_1})\circ(\Pi_{A_1})_\HH(\sigma_{A_1})^{-1})\\
&=(\FM{\kf})_\HH((\Pi_{A_1})_\HH(v_1\circ\sigma_{A_1}))\circ((\FM{\kf})_\HH((\Pi_{A_1})_\HH(\sigma_{A_1})))^{-1}\\
&=(\Pi_{A_2})_\HH((\FM{\kg})_\HH(v_1\circ\sigma_{A_1}))\circ((\Pi_{A_2})_\HH((\FM{\kg})_\HH(\sigma_{A_1})))^{-1}\\
&=(\Pi_{A_2})_\HH(v_2\circ(\FM{\kg})_\HH(\sigma_{A_1}))\circ((\Pi_{A_2})_\HH((\FM{\kg})_\HH(\sigma_{A_1})))^{-1}\\
&=(\Pi_{A_2})_\HH(v_2\circ(\FM{\ke})_\HH(\sigma_{A_1}))\circ((\Pi_{A_2})_\HH((\FM{\ke})_\HH(\sigma_{A_1})))^{-1}\\
&=\pi(v_2,(\FM{\ke})_\HH(\sigma_{A_1})).
\end{split}
\]
Theorem \ref{thm:Toda} concludes the proof.
\end{proof}

In general, even when we consider deformations of Kummer surfaces induced by those of the corresponding abelian surfaces, the converse of the previous result it is not expected to hold true.

\subsection{Enriques surfaces}\label{subsec:Enriques}

Let $Y$ be an Enriques surface, i.e.\ a minimal smooth projective surface with $2$-torsion canonical bundle $\omega_Y$ and $H^1(Y,\ko_Y)=0$. The universal cover $\pi:X\to Y$ is a K3 surface and it carries a fixed-point-free involution $\varpi:X\to X$ such that $Y=X/G$, where $G=\langle\varpi\rangle$.
In this special setting, $\coh(Y)$ is naturally isomorphic to the abelian category $\coh_G(X)$ which yields an equivalence $\Db(Y)\cong\Db_G(X)$, which will be tacitly meant for what follows.

Notice that, by functoriality, since $\pi$ is an \'etale morphism, we have an induced morphism $\pi^*:\HH^*(Y)\to\HH^*(X)$ which is compatible with the Hochschild--Kostant--Rosenberg isomorphism.

\begin{prop}\label{prop:Enriques}
Let $Y_1$ and $Y_2$ be Enriques surfaces, $\pi_i:X_i\to Y_i$ be their universal covers, and $v_i\in\HH^2(Y_i)$, for $i=1,2$. Then the following are equivalent:
\begin{itemize}
    \item[(i)] There exists a Fourier--Mukai equivalence $$\FM{\widetilde\ke}:\Db(Y_1,v_1)\isomor\Db(Y_2,v_2)$$ with $\widetilde\ke\in\Dp(Y_1\times Y_2,-J(v_1)\boxplus v_2)$.
    \item[(ii)] There exists an orientation preserving effective Hodge isometry $$g:\Htil(X_1,\pi_1^*(v_1),\ZZ)\isomor\Htil(X_2,\pi_2^*(v_2),\ZZ)$$ which is $G$-equivariant, i.e.\ $\varpi^*\circ g=g\circ\varpi^*$.
    \item[(iii)] There exists a Fourier--Mukai equivalence $$\FM{\widetilde\kf}:\Db(X_1,\pi_1^*(v_1))\isomor\Db(X_2,\pi_2^*(v_2))$$ with $\widetilde\kf\in\Dp(X_1\times X_2,-J(\pi_1^*(v_1))\boxplus \pi_2^*(v_2))$ such that $(\varpi\times\varpi)^*\kf\cong\kf$, where $\kf:=\mathbf{L}\iota^*\widetilde\kf$.
    \end{itemize}
\end{prop}

\begin{proof}
The equivalence of (ii) and (iii) is simply a rewriting of Theorem \ref{thm:main1} in the equivariant context using \cite[Prop.\ 3.4]{MMS} (or better its version for equivalences).
Now, due to \cite[Sect.\ 3.3]{Pl}, the existence of the kernel $\kf\in\Db(X_1\times X_2)$ of a Fourier--Mukai equivalence $\FM{\kf}$ such that $(\varpi\times\varpi)^*\kf\cong\kf$ is equivalent to the existence of an object $\ke\in\Db(Y_1\times Y_2)$ which gives rise to an equivalence $\FM{\ke}:\Db(Y_1)\isomor\Db(Y_2)$.
Hence, the equivalence of (i) and (iii) is a consequence of the commutativity of the following diagram
\[
\xymatrix{
    \HH^*(Y_1)\ar[rr]^{(\FM{\ke})^\HH}\ar[d]_{\pi_1^*}&&\HH^*(Y_2)\ar[d]^{\pi_2^*}\\
    \HH^*(X_1)\ar[rr]^{(\FM{\kf})^\HH}&&\HH^*(X_2).
    }
\]
This commutativity can be checked using the isomorphism of functors $\pi_2^*\circ\FM{\ke}\cong\FM{\kf}\circ\pi_1^*:\Db(Y_1)\to\Db(X_2)$, which in turn can be deduced from \cite[Sect.\ 7.3]{HFM} and \cite{Pl}.
\end{proof}

\medskip

{\small\noindent{\bf Acknowledgements.}  The final version of this
paper was written during the authors' stay at the
Hausdorff Center for Mathematics (Bonn) whose hospitality is
gratefully acknowledged. It is a pleasure to thank Daniel
Huybrechts for useful advice, Wendy Lowen and Yukinobu Toda for kindly answering our questions and the referee for helping us clarifying the motivations of the main result of this paper. We are grateful to Sukhendu Mehrotra for taking part in the discussion in which the main idea of this paper came to light and for writing the appendix. The second named author was partially supported by the Universit\`{a} degli Studi di Milano (FIRST 2004).
}

\newpage

\appendix

\section{Duality for Infinitesimal Deformations}\label{sec:Appendix}
\begin{center}
\small{by {\sc Sukhendu Mehrotra}}
\end{center}

\medskip

Let $X, Y$ be smooth and projective varieties, with
$(0,\beta,\gamma)\in\HT^2(X)$, $(0,\beta ',\gamma ')\in\HT^2(Y)$,
and suppose $f: (X, \ko_X^{(\beta,\gamma)}) \to
(Y,\ko_Y^{(\beta',\gamma')})$ is a morphism of locally ringed
spaces over $R_1:=\CC[\epsilon]/(\epsilon^2)$ (the notation here
is that of Section \ref{subsec:category}). Then, for $\aa ' \in
H^2(Y, \ko_Y)$, there is defined a push-forward functor between
twisted derived categories $\aR f_*:
\Db(\coh(\ko_X^{(\beta,\gamma)},\widetilde{f^*\alpha'})) \to
\Db(\coh(\ko_Y^{(\beta',\gamma')},\widetilde{\alpha'}))$
(\cite{Toda1}). The purpose of this appendix is to prove the
existence of a right adjoint $f^! :
\Db(\coh(\ko_Y^{(\beta',\gamma')},\widetilde{\alpha'})) \to
\Db(\coh(\ko_X^{(\beta,\gamma)},\widetilde{f^*\alpha'}))$ to $\aR
f_*$ under suitable hypotheses on $f$. More precisely, what is
proven is the existence of a ``dualizing complex.''

The setting here will be slightly more general than that of the
main article. To begin with, let $X$ denote a separated, finite
type scheme over $\CC$. By an {\em infinitesimal deformation} of
$X$, we will mean a locally ringed space $\tX=(X, \ka_X)$, where
$\ka_X$ is a sheaf of local flat $R_1$-algebras, with a fixed
isomorphism $\ka_X \otimes_{R_1} \CC \cong \ko_X$, such that
\begin{itemize}
\item[($\ast$)] \hspace{4.7cm} $0 \to \ko_X \to \ka_X \to \ko_X
\to 0$
\end{itemize}
is a central extension. In addition, we require that
the following condition be satisfied:
\begin{itemize}
\item[($\ast \ast$)] for any local section $f\in \ko_X(U)$ over an
affine open set $U$, there is a lift $\hat{f}\in \ka_X(U)$ such
that the multiplicative set $\{ \hat{f}^n : n\geq 0\}$ satisfies
the left and right Ore localization conditions (see \cite{MR},  2.1.6).
\end{itemize}
Naturally, infinitesimal deformations  then form a category in the
obvious way.

If $X$ is affine, and $M$ is a $\ka_X(X)$ module, we define the
{\em sheaf associated to} $M$ as the sheaf $M^{\sim}$ which on
principal opens $X_f$ has sections $M_{\hat{f}}$. A quick word about
this construction: If $X_g \subset X_f$,
we have that $g^n=f\alpha$, for some $n>0$ and $\alpha\in \ko_X(X)$.
Thus, by condition $(\ast)$, $\hat{g}^n$ and $\hat{f}\hat{\alpha}$ differ
by a central square-zero element, so that if the first is invertible, so is the
other. The universal property of localization (\cite{MR}) then yields a
restiction map $M_{\hat{f}} \to M_{\hat{g}}$; the uniqueness of this map
shows that as defined, $M^{\sim}$ is indeed a sheaf. The alert reader will
have noticed that we tacitly assumed a {\em choice} of a lift $\hat{f}$
for any given $f\in \ko_X(X)$ in making this definition. However, the same
reasoning as above shows that any other choice of lifts yields a
canonically isomorphic sheaf. Finally, let us mention that an important
property of this construction is that the assignment $M \mapsto M^{\sim}$ is
an exact functor (see propositions 2.1.16 (ii) and 2.1.17 (i) of \cite{MR}).

The categories $\coh(\tX):=\coh(\ka_X)$ and
$\qcoh(\tX):=\qcoh(\ka_X)$ of coherent and quasi-coherent sheaves
of {\em right} modules are defined in the usual way, in terms of
free presentations on open sets.  Using the exactness of the
associated-sheaf construction, it can easily be verified that the
analogue of Lemma 3.1 in \cite{Toda1}  holds: to wit, for any
$\kf \in \qcoh(\tX)$ and any affine open $U \subset X$, $\kf|_U =
\kf(U)^{\sim}$, the sheaf associated to the module $\kf(U)$, and
that the functor $M\mapsto M^{\sim}$ gives an equivalence
$\ka_{\tX}(U){\text{-}\cat{Mod}} \to \qcoh(\widetilde{U})$.

It is standard that the category $\mmod{\tX}$ of
$\ka_{\tX}$-modules has enough injective and flat objects
(\cite{Toda1}). So, for any morphism $f: \tX \to \tY$ of
infinitesimal deformations, the functors $f_*: \mmod{\tX} \to
\mmod{\tY}$ and $f^* : \mmod{\tY} \to \mmod{\tX}$ between
categories of right modules extend to derived functors $\aR f_* :
\D^+( \mmod{\tX}) \to \D^+ ( \mmod{\tY})$ and $\eL f^* :
\D^-(\mmod{\tY}) \to \D^-(\mmod{\tX})$. In fact, for our purposes,
we shall employ the general result  Theorem 4.5 of \cite{Sp} from
which it follows that $\aR f_*$ extends to a functor $\aR f_* :
\D( \mmod{\tX}) \to \D( \mmod{\tY})$ between {\em unbounded}
derived categories. Subtleties like the equivalence between
$\D_{\qcoh}( \mmod{\tX})$ and $\D(\qcoh(\tX))$ have been dealt
with in \cite{BN} and are not a cause for concern; it is then immediately
seen that $\aR f_*$ and $\eL f^*$ are defined between categories
of quasi-coherent sheaves.

The next couple of results establish the existence of a right
adjoint to the pushfoward functor for unbounded derived categories
of quasi-coherent sheaves. This is a direct application of
Neeman's general approach to duality via topological methods.

\begin{lem}\label{lemma:cop1}
Let $f: \tX \to \tY$ be a morphism of infinitesimal deformations,
with $X$ a separated scheme and $\fb:=f\otimes_{R_1} \CC$ a
separated morphism of schemes. Then, the functor $\aR f_* : \D(
\qcoh(\tX)) \to \D( \qcoh(\tY))$ respects coproducts, that is, for any small
set $\Lambda$,  the natural map
$\amalg_{\lambda \in \Lambda}\aR f_*(\kf_{\lambda}) \to
\aR f_*( \amalg_{\lambda \in \Lambda} \kf_{\lambda})$ is an isomorphism.
\end{lem}
\begin{proof}
The proof of Lemma 1.4 in \cite{N} carries over word-for-word.
\end{proof}

\begin{thm}\label{theorem:ajoint1} {\bf (Neeman)}
Under the hypotheses of the previous lemma, $\aR f_*$ admits a
right adjoint $f^! : \D( \qcoh(\tY)) \to \D( \qcoh(\tX))$.
\end{thm}

\begin{proof}
The proofs of Lemma 2.6 in \cite{N} and Proposition 6.1 in
\cite{BN} can be adapted to our setting to show that
if  $\kf \in \D(\qcoh(\tX))$ satisfies $\Hom_{\tX} (\kl, \kf) = 0$ for every
perfect object $\kl$, it must in fact be the zero object; in other words,
the category $\D(\qcoh(\tX))$ is
compactly generated. This, and the previous lemma,
furnish the hypotheses of Theorem 4.1 in \cite{N}, from which
the result follows.
\end{proof}

In the following paragraphs, we prove the claimed coherence and
boundedness properties of the right adjoint $f^!$. From now on,
all morphisms between infinitesimal deformations will be assumed
to be of {\em finite type}.

Denote by $\ix: X \to \tX$ and $\iy: Y \to \tY$ the canonical
immersions;  we observe that the functors  $(\ix)_*$ and $(\iy)_*$
are exact.

\begin{lem}\label{lemma:comm}
Let $f: \tX \to \tY$ be a flat morphism of infinitesimal
deformations, with $\fb: X \to Y$ a smooth morphism of schemes.
Then given $\kf \in \D(\qcoh(Y))$, there is an isomorphism $(\ix)_*\fb^!\kf
\cong  f^!(\iy)_*\kf$ in $\D(\qcoh(\tX)$.

\end{lem}

Assuming the previous lemma for the moment, we have the following
result.

\begin{prop}\label{prop:coh}
Let $f: \tX \to \tY$ be a flat morphism of infinitesimal
deformations, with $\fb$ smooth. Then, $f^!:   \D( \qcoh(\tY)) \to
\D( \qcoh(\tX))$ restricts to a functor $f^!: \Dp(\tY) \to \Db(
\coh(\tX))$, the category $\Dp(\tY)$ being that of perfect
complexes on $\tY$.
\end{prop}

\begin{proof} Given $\kf \in \Dp(\tY)$, one has the exact triangle:
$$
(\iy)_*\eL\iy^*\kf \mor[\cdot\epsilon]\kf\mor(\iy)_*\eL\iy^*\kf.
$$
Applying the functor $f^!$ and  rearranging, using Lemma
\ref{lemma:comm}, we get the exact triangle:
$$
(\ix)_*\fb^!(\eL\iy^*\kf) \to f^!\kf \to
(\ix)_*\fb^!(\eL\iy^*\kf).
$$
As $\fb^!$ exists between derived categories of {\em coherent}
sheaves and $\eL\iy^*\kf \in \Db(\coh(\tX))$, the outer two terms
are in $\Db(\coh(\tX))$. Consequently, the middle term is also.
\end{proof}

\noindent{\it Proof of Lemma \ref{lemma:comm}.} Pick any object
$\kg \in \D^-(\qcoh(\tX))$ and $\kf \in \D(\qcoh(Y))$. It follows
by adjunction that
$$\aR\Hom_{\tX}(\kg, (\ix)_*\fb^!\kf)\cong \aR\Hom_{\tY}(\aR \fb_* \eL\ix^* \kg, \kf).$$
Assume the base-change property which will be proved in Lemma
\ref{lemma:kuz}: $\eL\iy^*\aR f_* \cong \aR\fb_* \eL\ix^*$, that
is, the standard natural transformation of functors is an
isomorphism. Then, continuing from the previous line, we find that
\[
\begin{split}
\aR\Hom_{\tX}(\kg, (\ix)_*\fb^!\kf)&\stackrel{\cong}{\to} \aR\Hom_{\tY}(\eL\iy^*\aR f_*\kg, \kf)\\
                                                 &\cong \aR\Hom_{\tX}(\kg, f^!(\iy)_*\kf),
\end{split}
\]
all isomorphisms being canonical and functorial in the first argument; call their
composition $can$. To simplify typography, set $\ka=(\ix)_*\fb^!\kf$ and
$\kb:=  f^!(\iy)_*\kf$. Express $\ka$
as the homotopy colimit (see Remark 2.2, \cite{BN}) :
$$
\amalg_{i>0} \left(\ka\right)^{\leq i} \stackrel{id-s}{\longrightarrow}  \amalg_{i>0}
\left(\ka\right)^{\leq i} \longrightarrow \ka
$$
where the bounded above complex $\left(\ka\right)^{\leq i}$ is the $t$-th truncation:
$$
\kh^n  \left( \left(\ka\right)^{\leq i} \right) = \left\{
\begin{array}{l}
\kh^n \left( \ka\right), \mbox{\; if \;} n\leq i \\
\\
0, \mbox{\; if \;} n>i,
\end{array} \right.
$$
and $s$ is the direct sum of the morphisms $\left(\ka\right)^{\leq i} \to
\left(\ka\right)^{\leq i+1}$. This yields the following diagram with exact triangles for rows:
\[
\xymatrix{
    \aR\Hom(\ka, \ka)\ar[r]&
   \Pi_{i>0}\aR\Hom\left(\left(\ka\right)^{\leq i}, \ka\right)\ar[r]\ar[d]^{\stackrel{\cong}{can}}&
    \Pi_{i>0}\aR\Hom\left(\left(\ka\right)^{\leq i}, \ka\right) \ar[d]^{\stackrel{\cong}{can}}\\
    \aR\Hom(\ka, \kb)\ar[r]&
   \Pi_{i>0}\aR\Hom\left(\left(\ka\right)^{\leq i}, \kb\right)\ar[r]&
    \Pi_{i>0}\aR\Hom\left(\left(\ka\right)^{\leq i}, \kb\right).
    }
\]
 By the axioms defining a triangulated category, there exists an
isomorphism $\aR\Hom(\ka, \ka) \dashrightarrow\aR\Hom(\ka, \kb)$ making the diagram commutative;
let $\mu \in \Hom(\ka,\kb)$ be the image of the identity in  $\Hom(\ka,\ka)$ given by this
isomorphism. Any morphism $\kl \to \ka$, with $\kl$ a perfect complex, factors through
$\amalg_{i>0} \left(\ka\right)^{\leq i} \to \ka$. Thus, by the commutativity of the diagram above,
one sees that $can_{\kl}: \aR\Hom_{\tX}(\kl, \ka) \stackrel{\cong}{\to}
\aR\Hom_{\tX}(\kl, \kb)$
coincides with $\mu_*$. Since $\D(\qcoh(\tX)$ is compactly generated by the proof of Theorem
\ref{theorem:ajoint1}, this proves that $\mu: (\ix)_*\fb^!\kf \to f^!(\iy)_*\kf$
is an isomorphism.

\qed

\medskip

We will need some notation for what follows. Given infinitesimal
deformations $\tX$ and $\tY$, let $\tX \times \tY$ denote the
infinitesimal deformation $(X \times Y, \ka_{X \times Y})$, where
the restriction of the structure sheaf $\ka_{X \times Y}$ on open
sets of the form $U \times V \subset X \times Y$, with $U$ an open
affine in $X$ and $V$ an open affine in $Y$, is $(\ka_X(U)
\otimes_{R_1} \ka_Y(V))^{\sim}$ (observe that $ \ka_{X \times Y}$
satisfies $(\ast \ast)$ by Lemma 2.1.8 in \cite{MR}). Suppose $f
:\tX \to \tY$ is a flat morphism; write $\Gg_{\iy}: Y \to Y \times
\tY$ for the graph of $\iy$, $\Gg_f : \tX \to \tY \times \tX$ for
the graph of $f$, and $\dd : X \to Y \times \tX$ for the morphism
given by the pair $(\fb, \ix)$. Define $\kk_{\iy} := (\Gg_{\iy})_*
\ko_Y$, $\kk_f := (\Gg_f)_*\ka_X$, and $\kk := \dd_*\ko_X$. Then
one has isomorphisms of functors between bounded above derived
categories of quasi-coherent sheaves :
$$
\FM{\kk_{\iy}} \cong \eL\iy^*, \hspace{10mm} \FM{\kk_f} \cong \aR
f_*, \hspace{10mm} \FM{\kk} \cong \aR\fb_*\eL\ix^*.
$$
Note that in defining the Fourier--Mukai functors (see
\cite{Toda1}) $\FM{\kk_{\iy}}$, $\FM{\kk_f}$ and $\FM{\kk}$, we
have used the {\em bimodule} structure of the kernels for the
respective structure sheaves. Consider this Cartesian diagram of
infinitesimal deformations:
\[
\xymatrix{
    X\ar[rr]^{\ix}\ar[d]_{\fb}&&\tX \ar[d]^{f}\\
    Y\ar[rr]^{\iy}&&\tY.
    }
\]
The following observation about this diagram will be used later:
there is an isomorphism of functors \begin{equation}\label{obs}
f^*(\iy)_* \cong (\ix)_*\fb^* .
\end{equation}

\begin{lem}\label{lemma:kuz}{\bf (Kuznetsov)}
There is an isomorphism $\kh^0(\kk_f \circ \kk_{\iy}) \cong \kk$,
inducing a morphism of functors  $\eL\iy^* \aR f_* \cong
\FM{\kk_{\iy}}\circ\FM{\kk_f} \to \FM{\kk} \cong
\aR\fb_*\eL\ix^*$. This morphism is an isomorphism if and only if
$\kk_f \circ \kk_{\iy} \cong \kk$, which in turn holds true if and
only if $ f^*(\iy)_* \cong (\ix)_*\fb^* $.
\end{lem}

\begin{proof}
The first statement follows essentially from the proof of Lemma
2.17 of \cite{Kuz}. The reader should note that while our
definition of Fourier--Mukai transform is different from
Kuznetsov's, the proof still works because of the isomorphism
(\ref{obs}).

For the proof of the second statement, we would like to invoke
Proposition 2.15 of \cite{Kuz}, namely, that an isomorphism of
kernel functors induces an isomorphism of kernels. This result,
however, depends on the existence of a perfect spanning class (see
Definition 2.9 in \cite{Kuz}), which is false for a general
infinitesimal deformation. Nevertheless, the category
$\D^-(\qcoh(\tX))$ certainly does admit a ``flat spanning class'':
$\{ \ka_x:\, x \in X \}$, where $\ka_x$ is the stalk of $\ka_X$ at
$x$, seen as a quasi-coherent sheaf over $\tX$. Indeed, flatness
follows from Proposition 2.1.16 (ii) of \cite{MR}, whereas the
spanning property, that is $H^{\bullet}(X, \kg \otimes_{\ka_X}
\ka_x)=0$ for all $x \in X$ implies $\kg=0$, can be proven by the
same argument as in Lemma 2.13 of \cite{Kuz}, using flatness and
the crucial observation that $H^{i>0}(M_x)=0$ for any
quasi-coherent $\ka_X$-module $M$. This is enough to prove the
analogue of Lemma 2.15  for bounded above derived categories of
quasi-coherent sheaves, an easy exercise left to the reader.

The last statement is obtained simply by interpreting the kernels
in the lemma as kernels of Fourier--Mukai functors going in the
opposite direction.
\end{proof}

In view of observation (\ref{obs}), we in fact  have $\eL\iy^* \aR
f_* \cong \aR\fb_*\eL\ix^*$, thus completing the proof of Lemma
\ref{lemma:comm}.

\begin{thm}
Let $f :\tX \to \tY$ be a flat morphism, with $\fb$ smooth. Then
$f^!$ commutes with coproducts and is given by the formula:
$$
f^!\kf= f^*\kf\otimes_{\ka_X} f^!\ka_Y.
$$
\end{thm}

\begin{proof}
First note that the functors $\eL\iy^*$, $\fb^!$ and $(\ix)_*$
commute with coproducts. Then, using the second exact triangle in
the proof of Proposition \ref{prop:coh}, the first part of the
result follows. For the second part, we appeal to the proof of
Theorem 5.4 of \cite{N}, making use of the fact that $f^!(\ka_Y)
\in \Db(\coh(\tX))$ by Proposition \ref{prop:coh}, so that the
necessary tensor products appearing there are defined (the point
is that tensor products need not be defined for the unbounded
derived category in the noncommutative setting).
\end{proof}

\begin{cor}\label{cor:FMinv}
Let  $\FM{\ke} : \Db(\qcoh(\tX)) \to \Db(\qcoh(\tY))$ be an
integral transform, with $\ke$ a perfect object. Then its right
adjoint is an integral functor with kernel $p_{\tY}^!(\ka_X)
\otimes \ke^{\vee}$, where $p_{\tY} : \tX \times \tY \to \tY$ is
the second projection.  In particular, if $X$ and $Y$ are smooth
and projective, the inverse of a Fourier--Mukai equivalence
between $\Db(\coh(\tX))$ and $\Db(\coh(\tY))$ is a Fourier--Mukai
functor.
\end{cor}

Denote by $\kz(\ka_X)$ the center of $\ka_X$, and for $\tt \in
H^2(\tX, \kz(\ka_X)^{\times})$, let $\qcoh(\tX, \tt)$ and
$\coh(\tX, \tt)$ be the categories of twisted quasi-coherent and,
respectively, twisted coherent sheaves on X (see Definition 4.1,
\cite{Toda1}).

\begin{cor}\label{cor:twisted}
Let $f: \tX \to \tY$ be a flat morphism, with $\fb$ smooth and
projective; let $\tt \in H^2(\tY, \kz(\ka_Y)^{\times}).$
 Then, $\aR f_*: \D(\qcoh(\tX, f^*\tt)) \to \D(\qcoh(\tY, \tt))$ admits a right adjoint $f^!$, with
 $f^!(\kf) = f^*(\kf) \otimes_{\ka_X} f^!(\ka_Y)$, where $f^!$ in the term on the right is the untwisted right
adjoint of Theorem \ref{theorem:ajoint1}. Moreover, $f^!$
restricts to $f^!: \Dp(\tY, \tt) \to \Db(\coh(\tX), f^*\tt)$.

The analogue of Corollary \ref{cor:FMinv} also holds in this
setting.
\end{cor}
\begin{proof}
This is standard given the previous results in the untwisted setting (see, for example, Theorem 2.4.1 of \cite{Cal} ).
\end{proof}

\begin{remark}
Let us define {\em deformations of order} $n$ as locally ringed
spaces of the form $(X, \ka_{X,n})$, where $X$ is a separated,
noetherian scheme over $\CC$, and $\ka_{X,n}$ is a sheaf of  local
flat $R_n:=\CC[\epsilon]/\CC[\epsilon^{n+1}]$-algebras with a
fixed isomorphism $\ka_X \otimes_{R_n} \CC \cong \ko_X$, satisfying
the following recursive conditions:
\begin{itemize}
\item[($\ast '$)]  $0 \to \ko_X \to \ka_{X,n} \to \ka_{X,n-1}
\to 0$ is a central extension (with $\ka_{X,n-1}$ a deformation
of order $n-1$),
\item[($\ast \ast '$)] for any local section $f\in \ko_X(U)$ over an
affine open set $U$, there is a lift $\hat{f}\in \ka_{X,n}(U)$ such
that the multiplicative set $\{ \hat{f}^n : n\geq 0\}$ respects
the left and right Ore localization conditions.
\end{itemize}
Then, it is not hard to see by d\'evissage, using the exact
triangles arising from the exact sequence
$
0\to \CC \to R_{i+1} \to R_i \to 0,
$
as in the proof of Proposition \ref{prop:coh}, that the results of
this appendix carry over to order $n$. The usefulness of this
generalization may be limited, however.
Indeed, while every infinitesimal
deformation of the abelian category $\coh(\ko_X)$ for a smooth and
projective variety $X$ admits a description as the twisted category
$\coh(\tX, \tau)$ of an infinitesimal deformation $(\tX, \ka_X)$ in our
sense, it
is not clear to us what form a general higher order deformation takes.
\end{remark}

\end{document}